\UseRawInputEncoding
\documentclass[11pt,reqno]{amsart}
\RequirePackage[OT1]{fontenc}
\usepackage{palatino}
    \usepackage{graphicx}   
    \usepackage{multirow}   
    \usepackage{amsthm,amsmath,amsfonts,amssymb,color,stmaryrd,dsfont} 
    \usepackage{bm}
\usepackage{ulem}
\usepackage{tikz}
\usepackage{hyperref,url}
\usepackage{graphicx} 
\usepackage{stmaryrd}
\usepackage{float} 
\usepackage{subfigure,caption} 
\usepackage[a4paper, total={6.7in, 8.5in}]{geometry} 

\begin{document} 

\newcommand{\bbA}{{\bf A}}
	\newcommand{\cbbA}{{\check \bbA}}
	\newcommand{\mbbA}{{\mathcal A}}
	\newcommand{\bbbA}{{\bar \mbbA}}
	\newcommand{\hbbA}{{\widehat \bbA}}
	\newcommand{\tbbA}{{\tilde \bbA}}
	\newcommand{\bba}{{\bf a}}
	\newcommand{\bbB}{{\bf B}}
	\newcommand{\bbC}{{\bf C}}
	\newcommand{\bbc}{{\bf c}}
	\newcommand{\bbD}{{\bf D}}
	\newcommand{\bbd}{{\bf d}}
	\newcommand{\bbe}{{\bf e}}
	\newcommand{\bbE}{{\bf E}}
	\newcommand{\rE}{{\rm E}}
	\newcommand{\bbf}{{\bf f}}
	\newcommand{\bbF}{{\bf F}}
	\newcommand{\bbP}{{P}}
	\newcommand{\bbg}{{\bf g}}
	\newcommand{\bbG}{{G}}
	\newcommand{\bbK}{{\bf K}}
	\newcommand{\bbH}{{H}}
	\newcommand{\bbh}{{\bf h}}
	\newcommand{\bbw}{{\bf w}}
	\newcommand{\bbI}{{\bf I}}
	\newcommand{\bbi}{{\bf i}}
	\newcommand{\bbj}{{\bf j}}
	\newcommand{\bbJ}{{\bf J}}
	\newcommand{\bbk}{{\bf k}}
	\newcommand{\bbl}{{\bf 1}}
	\newcommand{\bbM}{{\bf M}}
	\newcommand{\bbm}{{\bf m}}
	\newcommand{\bbN}{{\bf N}}
	\newcommand{\bbn}{{\bf n}}
	\newcommand{\bbQ}{{Q}}
	\newcommand{\bbq}{{\bf q}}
	\newcommand{\bbO}{{\bf O}}
	\newcommand{\bbR}{{\bf R}}
	\newcommand{\bbr}{{\bf r}}
	\newcommand{\bbs}{{\bf s}}
	\newcommand{\cbbs}{{\check \bbs}}
	\newcommand{\hbbs}{{\hat \bbs}}
	\newcommand{\tbbs}{{\tilde \bbs}}
	\newcommand{\bbS}{{\bf S}}
	\newcommand{\cbbS}{{\check \bbS}}
	\newcommand{\wB}{{\widehat B}}
	\newcommand{\tbbS}{\widetilde{\bf S}}
	\newcommand{\hbbS}{\widehat{\bf S}}
	\newcommand{\obbS}{\overline{\bf S}}
	\newcommand{\bbt}{{\bf t}}
	\newcommand{\bbT}{{\bf T}}
	\newcommand{\hbbT}{\widetilde{\bf T}}
	\newcommand{\tbbT}{\widetilde{\bf T}}
	\newcommand{\obT}{{\overline{\bf T}}}
	\newcommand{\bbU}{{\bf U}}
	\newcommand{\bbu}{{\bf u}}
	\newcommand{\bbV}{{\bf V}}
	\newcommand{\bbv}{{\bf v}}
	\newcommand{\tbbv}{\widetilde{\bf v}}
	\newcommand{\bbW}{{W}}
	\newcommand{\tbbW}{\widetilde{\bf W}}
	\newcommand{\hbbW}{\widehat{\bf W}}
	\newcommand{\bbX}{{X}}
	\newcommand{\tbby}{\tilde {\bf y}}
	\newcommand{\tbbX}{\widetilde {\bf X}}
	\newcommand{\hbbX}{\widehat {\bf X}}
	\newcommand{\bbx}{{\bf x}}
	\newcommand{\obbx}{{\overline{\bf x}}}
	\newcommand{\tbbx}{{\widetilde{\bf x}}}
	\newcommand{\hbbx}{{\widehat{\bf x}}}
	\newcommand{\obby}{{\overline{\bf y}}}
	\newcommand{\hbbY}{{\widehat{\bf Y}}}
	\newcommand{\tbbY}{{\widetilde{\bf Y}}}
	\newcommand{\bbY}{{\bf Y}}
	\newcommand{\bby}{{\bf y}}
	\newcommand{\bbZ}{{\bf Z}}
	\newcommand{\bbz}{{\bf z}}
	\newcommand{\bbb}{{\bf b}}
	\newcommand{\cD}{{\cal D}}
	\newcommand{\bbL}{{\bf L}}
	\newcommand{\bxi} {\boldsymbol  \xi}
	\newcommand{\bbeta} {\boldsymbol  \eta}
	\newcommand{\utm}{\underline{ \tilde m}}
	\newcommand{\um}{\underline{m}}
	\newcommand{\la}{\langle}
	\newcommand{\ra}{\rangle}
	\newcommand{\bla}{\big{\langle}}
	\newcommand{\bra}{\big{\rangle}}
	\newcommand{\Bla}{\Big{\langle}}
	\newcommand{\Bra}{\Big{\rangle}}
	
	\newcommand{\rdd}{\textcolor{red}}
	\newcommand{\bll}{\textcolor{blue}}
	
	\newcommand{\md}{\mbox{d}}
	\newcommand{\non}{\nonumber\\}
	\newcommand{\tr}{{\rm tr}}
	\newcommand{\Tr}{{\rm Tr}}
	\newcommand{\E}{{\mathbb{E}}}
	\newcommand{\rP}{{\mathbb{P}}}
	\newcommand{\bqa}{\begin{eqnarray}}
		\newcommand{\eqa}{\end{eqnarray}}
	
	\newcommand{\bqn}{\begin{eqnarray*}}
		\newcommand{\eqn}{\end{eqnarray*}}
		
	\theoremstyle{plain}
	\newtheorem{thm}{Theorem}[section]
	\newtheorem{corollary}[thm]{Corollary}
	\newtheorem{defin}[thm]{Definition}
	\newtheorem{prop}[thm]{Proposition}
	\newtheorem{remark}[thm]{Remark}
	\newtheorem{lemma}[thm]{Lemma}
	\newtheorem{assumption}[thm]{Assumption}
          \allowdisplaybreaks[4]


\newpage

\begin{center}
\large\bf
Extreme eigenvalues of Log-concave Ensemble
\end{center}

\vspace{0.5cm}
\renewcommand{\thefootnote}{\fnsymbol{footnote}}
\hspace{5ex}	
\begin{center}
 \begin{minipage}[t]{0.35\textwidth}
\begin{center}
Zhigang Bao\footnotemark[1]  \\
\footnotesize {Hong Kong University of Science and Technology}\\
{\it mazgbao@ust.hk}
\end{center}
\end{minipage}
\hspace{8ex}
\begin{minipage}[t]{0.35\textwidth}
\begin{center}
Xiaocong Xu\footnotemark[3]  \\
\footnotesize {Hong Kong University of Science and Technology}\\
{\it xxuay@connect.ust.hk}
\end{center}
\end{minipage}
\end{center}

\footnotetext[1]{Supported by  Hong Kong RGC grant  GRF 16303922 and NSFC grant 12271475 and NSFC22SC01}
\footnotetext[3]{Supported by  Hong Kong RGC grant GRF 16301520 and GRF 16305421}
\vspace{0.8cm}
\begin{center}
 \begin{minipage}{0.8\textwidth}\footnotesize{
Abstract:  In this paper, we consider the log-concave ensemble of random matrices, a class of covariance-type matrices $XX^*$ with isotropic log-concave $X$-columns. A main example is the covariance estimator of the uniform measure on isotropic convex body. Non-asymptotic estimates and first order asymptotic  limits for the extreme eigenvalues have been obtained in the literature. In this paper, with the recent advancements on log-concave measures \cite{chen, KL22},  we take a step further to locate the eigenvalues with a nearly optimal precision, namely, the spectral rigidity of this ensemble is derived.  Based on the spectral rigidity and an additional ``unconditional"  assumption, we further derive the Tracy-Widom law for the extreme eigenvalues of $XX^*$, and the Gaussian law for the extreme eigenvalues in case strong spikes are present. }
\end{minipage}
\end{center}


\thispagestyle{headings}
\section{Introduction and main results}\label{Sec. Intro and main}
\subsection{Background and matrix model}
Let $q_1,\cdots,q_N \in \mathbb{R}^{M}$ be i.i.d. isotropic log-concave random vectors. That means, $\E q_i q_i^* = I$ and $q_i$ possesses a density function $\exp (-V)$ where $V$ is convex. The commonly considered sample covariance matrix of $q_i$'s is defined as 
\begin{align}\label{def of S}
H\equiv H_N :=\frac{1}{N}  \sum_{i=1}^N q_i q_i^* =: XX^*,
\end{align}
with $X := (x_1,\cdots,x_N) = N^{-1/2}(q_1,\cdots,q_N)$ is the scaled random matrix. 
The study of the random matrix $H$ dates back to \cite{KLS}, which was motivated by the question `` How many independently and uniformly chosen points on an isotropic convex body are needed for the empirical covariance estimator to approximate the identity sufficiently well?". The question then boils down to the estimate of the extreme eigenvalues of the random matrix $H$, since uniform measure on isotropic convex body is isotropic log-concave. The problem has been investigated in \cite{Bou, GHT, Pao, Rud, Aub, Men}, and the strongest result is obtained in \cite{ALPT}. It is shown in \cite{ALPT} that in case $M\leq N\leq \exp(\sqrt{M})$, one has
\begin{align*}
1-C\sqrt{\frac{M}{N}}\log \frac{2N}{M}\leq \lambda_M(H)\leq \lambda_1(H) \leq 1+C\sqrt{\frac{M}{N}}\log \frac{2N}{M}
\end{align*}
with probability greater than $1-\exp(-c\sqrt{n})$, with numerical constants $C,c>0$. Here we ordered the eigenvalues in descending order $\lambda_1(H)\geq\cdots\geq \lambda_M(H)$. The  above non-asymptotic bounds also match the asymptotic estimates of the extreme eigenvalues in the regime $M\equiv M(N)$ and $M/N\equiv y\to y^\infty\in (0,\infty)$ in \cite{CT}, which read 
\begin{align}
\lambda_{M\wedge N}(H) \stackrel{\mathbb{P}} \longrightarrow (1-\sqrt{y^\infty})^2=:\lambda_-^{\infty}, \qquad \lambda_1(H) \stackrel{\mathbb{P}} \longrightarrow (1+\sqrt{y^\infty})^2=: \lambda_+^{\infty}, \qquad \text{as} \quad N\to \infty,  \label{converg}
\end{align}
where the first convergence is proved in the regime $y^\infty\in (0,1)$. 
Prior to the above convergence of the extreme eigenvalues, the Machenko-Pastur law (MP law) for $H$ was proved  in \cite{PP07}. Specifically, denote by $\nu_{1N}:=\frac{1}{M}\sum_{i=1}^M \lambda_i(H)$ the empirical spectral distribution of $H$, the results in \cite{PP07} states that as $N\to \infty$, the measure $\nu_{1N}$ converges weakly (a.s.) to 
\begin{align}
\nu_{y^\infty,1}({\rm d}x):=\frac{1}{2\pi xy^\infty}\sqrt{\big((\lambda_+^\infty-x)(x-\lambda_-^\infty)\big)_+}{\rm d}x+\Big(1-\frac{1}{y^\infty}\Big)_+\delta({\rm d}x)  \label{global law}
\end{align}
In the sequel, we will use the notations $\lambda_+, \lambda_-, \nu_{y,1}$ which can be obtained via replacing $y^{\infty}$ by the non-asymptotic parameter $y=\frac{M}{N}$ in the definitions of $\lambda_+^\infty, \lambda_-^\infty, \nu_{y^\infty,1}$. 

In this paper, we take a step further to study the asymptotic behavior of the eigenvalues of $H$, especially the extreme eigenvalues, on a local scale. The MP law in (\ref{global law}) depicts the global behaviour of the eigenvalues. A control of the eigenvalue locations down to the scale slightly above eigenvalue spacings is called ``local law". A systematic  study of local law 
started from \cite{ESY}. We also refer to \cite{erdHos2013local, ESY2, ESY3, EYY}, the survey \cite{BK} and more references therein for the local semicircle law for Wigner type matrices.  For the covariance type of matrices, the local MP type laws have been obtained in \cite{BPZ2, TV, YinPi,alex2014isotropic, knowles2017anisotropic}.   A direct consequence of local law is the spectral rigidity, which can be regarded as a precise large deviation control of the individual eigenvalues. The local law and spectral rigidity serve as the main technical inputs for the universality of local statistics. Especially, based on the dynamic approach and the Green function comparison approach, equipped with the local laws, the extreme eigenvalues of many random matrix models have been proved to follow the Tracy-Widom law, which is also known as the edge universality. For instance, we refer to \cite{EYY, erdHos2012spectral, lee2015edge, alt2020correlated,  huang2020transition, LY, LeeYin, knowles2017anisotropic} and also the monograph   \cite{EY} for the development on these approaches. Specifically, for covariance type of matrices, the local law and Tracy-Widom limit have been obtained in \cite{YinPi,YinPib, BPZ, BPZ2, DY, DY1} for sample covariance matrices and sample correlation matrices. In this work, we extend this line of research to the log-concave ensemble $H$. 

A major technical input for the proof of local laws is a Hanson-Wright type inequality, namely, a concentration inequality for the quadratic forms of a column or row of the random matrices. Thanks to the recent breakthrough \cite{chen, KL22} on the lower bound of the isoperimetric coefficient in the KLS conjecture, a sufficiently sharp Hanson-Wright type inequality is possible for the log-concave ensemble; see Lemma \ref{Large deviation lemma} below. Based on this concentration inequality, the proof of the local law and spectral rigidity then follows the  strategy in \cite{YinPi}. Nevertheless, the proof of the edge universality requires one to control high-order correlations among the matrix entries, which brings the major obstacle in contrast to the independent case in \cite{YinPi} or self-normalized independent case in \cite{YinPib, BPZ}. Instead of the general log-concave ensemble, we consider the  sub-class  of {\it unconditional} log-concave ensemble, for the edge universality. Further, again for the unconditional ensemble, we consider the case when the isotropic covariance is perturbed by a finite rank of spikes, and derive the Gaussian law for the extreme eigenvalues when the spikes are sufficiently large. 

In the next subsection, we state our main results.

\subsection{Main results}

Throughout the paper, our basic assumption is as follows.
\begin{assumption} \label{assu.1.1}We make the following two assumptions on $H=XX^*$.

(i){\it (On X)} We assume that $X$ has i.i.d. columns $x_i$'s, and $x_i$'s are isotropic log-concave.

(ii) {\it (On $y$)} We assume that $M/N\equiv y\to y^\infty \in (0,\infty)\setminus \{1\}$, as $N\to \infty$. 
\end{assumption}

To state our main results, we first introduce some basic notions that are commonly used throughout the whole paper.
We first introduce the notions of stochastic domination which originated from \cite{erdHos2013averaging}.
 
 \begin{defin}[Stochastic domination] \label{def.sd}
 	Let 
 	$$
 		\mathsf{X}=\left(\mathsf{X}^{(N)}(u): N \in \mathbb{N}, u \in \mathrm{U}^{(N)}\right), \mathrm{Y}=\left(\mathsf{Y}^{(N)}(u): N \in \mathbb{N}, u \in \mathsf{U}^{(N)}\right)
 	$$
 	be two families of random variables, where $\mathsf{Y}$ is nonnegative, and $\mathsf{U}^{
(N)}$ is a possibly $N$-dependent parameter set.

We say that $\mathrm{X}$ is stochastically dominated by $\mathsf{Y}$, uniformly in $u$, if for all small $\epsilon > 0$  and large $D > 0$ ,
$$
\sup _{u \in \mathsf{U}^{(N)}} \mathbb{P}\left(\left|\mathsf{X}^{(N)}(u)\right|>N^{\varepsilon} \mathsf{Y}^{(N)}(u)\right) \leqslant N^{-D}
$$
for large enough $N > N_0(\epsilon, D)$. If $\mathsf{X}$ is stochastically dominated by $\mathsf{Y}$, uniformly in $u$, we use the notation
$\mathsf{X} \prec \mathsf{Y}$ , or equivalently $\mathsf{X} = O_{\prec}(\mathsf{Y})$. Note that in the special case when $\mathsf{X}$ and $\mathsf{Y}$ are deterministic, $\mathsf{X} \prec \mathsf{Y}$ means that for any given $\epsilon > 0$, $|\mathsf{X}^{(N)}(u)| \le N^{\epsilon}\mathsf{Y}^{(N)}(u)$ uniformly in $u$, for all sufficiently large $N \ge N_0(\epsilon)$.
\end{defin}

Our first result is on the rigidity of the eigenvalues. 
Let the classical locations of the eigenvalues $\gamma_j$'s be defined  as follows:
\begin{align}\label{def of gammaj}
	\int_{\gamma_j}^{\lambda_{+}} \nu_{y,1} ({\rm d}x) = \frac{j-1/2}{N}.
\end{align}

\begin{thm} [Spectral rigidity] \label{rigidity} Suppose that Assumption \ref{assu.1.1} holds. 
 For any $1 \le j \le M\wedge N$,  we have 
		\begin{align}\label{eigenvalue rigidity}
			|\lambda_{j}(H) - \gamma_{j}| \prec N^{-2/3} (\min \{ M\wedge N+1-j,j \})^{-1/3}.
		\end{align}
\end{thm}

Next, we turn to the limiting distribution of the extreme eigenvalue. We need a further assumption on the ensemble. 

\begin{defin}[Unconditional random vector] \label{def.unconditional}
	A random vector $\xi=(\xi_1,\ldots, \xi_n)$ is unconditional if 
\begin{align*}
	(\xi_1,\cdots,\xi_n) \stackrel {d} =(\delta_1\xi_1,\cdots,\delta_n\xi_n) 
\end{align*}
for any choice of signs $\delta_1, \ldots, \delta_n$ independent of $\xi$. 
\end{defin}
	For any Hermitian matrix $A\in \mathbb{C}^{n\times n}$ we use $\lambda_1(A)\geq \dots\geq \lambda_n(A)$ to denote the ordered eigenvalues of $A$. Let $X^{\mathbf{w}}$ be an $M \times N$ random matrix with i.i.d. mean $0$ and variance $N^{-1}$ Gaussian entries, and the corresponding sample covariance matrix $W := X^{\mathbf{w}}(X^{\mathbf{w}})^*$ is known as the real Wishart matrix.  Denote $\mathbb{P}^W$ as the probability measure according to which the entries of $X^{\mathbf{w}}$ are distributed. $\mathbb{P}^H$ is defined analogously. We then have the following result
		
\begin{thm}[Universality of largest eigenvalue] \label{thm.universality}
Suppose that Assumption \ref{assu.1.1} holds. Assume further that $x_i$'s are distributed unconditionally. Then
	there exists $c > 0$ and $\delta > 0$ such that for any real number $s$, we have for sufficiently large $N$
	\begin{align*}
		\lim_{N \to \infty} \mathbb{P}^{H}\left( N^{2/3}(\lambda_1(H) - \lambda_{+} ) \le x \right) = \lim_{N \to \infty} \mathbb{P}^{W}\left(N^{2/3}(\lambda_1(W) - \lambda_{+} ) \le x  \right),
	\end{align*}
	where $\lambda_{+} = (1+ \sqrt{y})^2$. 
\label{universality thm}
\end{thm}
It is well known that for Wishart matrix, after appropriate centering and rescaling, the largest eigenvalue  converges in distribution to the Tracy-Widom law. Combining with Theorem \ref{universality thm}, we have the following corollary.
\begin{corollary} \label{cor.universality}
	Under the same condition as Theorem \ref{universality thm}, we have
	\begin{align*}
		\frac{N\lambda_1(H) - (\sqrt{M} + \sqrt{N})^2}{(\sqrt{M} + \sqrt{N})(1/\sqrt{M} + 1/\sqrt{N})^{1/3}} \to  \mathrm{TW}_1,
	\end{align*}
	where $\mathrm{TW}_1$ denotes the Tracy-Widom distribution.
\end{corollary}

\begin{remark} Our discussions on Theorem \ref{thm.universality} and Corollary \ref{cor.universality} can be extended to the smallest eigenvalue as well. Hence, the edge universality also holds at the left edge. For brevity, here we present and proof  the results for the largest eigenvalue only. 
\end{remark}

	Our last result is on the spiked model, which is defined as 
\begin{align}\label{def of Q}
	Q\equiv Q_N  := TXX^*T^*
\end{align}
 with the population covariance matrix 
\begin{align*}
	\Sigma := TT^* = I_{M} + \sum_{i=1}^r d_i \mathrm{e}_{i}\mathrm{e}_{i}^*.
\end{align*}
Here $\{\mathrm{e}_{i}\}$ is the standard basis and $r$ is a fixed number. We  further assume $d_1  \ge \cdots \ge d_r > \sqrt{y}$. The spiked model has been studied extensively in the literature. The primary interest is to investigate the limiting behaviour of its extreme eigenvalues and eigenvectors. Not trying to be comprehensive, we refer to \cite{BBP, BY, BN, CD, BKYPCA, Joh, BDWW, BW} for related study.

We summarize the more precise assumptions on $d_i$'s as follows.
\begin{assumption}\label{assume on spikes}
	Assume $r$ to be fixed. For any $1 \le i, j \le r$, we assume $d_i  \ge \sqrt{y}+\epsilon$ and  $\min_{i \neq j}|d_i - d_j| > N^{-1/2+\epsilon}$ for some small constant $\epsilon > 0$.
\end{assumption}
Under the above assumption, we have the following result. 
\begin{thm}\label{Spiked model distribution}
	Under Assumption \ref{assume on spikes} and the  assumptions in Theorem \ref{universality thm}, there exists a random variable $\Phi_i$ such that the outlying eigenvalue admits the expansion
	\begin{align}
		\lambda_i(Q) = \theta(d_i) + \sqrt{\frac{d_i^2 - y}{N}}\Phi_i + O_{\prec}\left( \frac{1}{N^{1/2+\epsilon}} \right), \quad i = 1,\cdots,r,
	\end{align}
	for some small constant $\epsilon > 0$, and $\theta(d_i) = 1 + d_i + y + yd_i^{-1}$. Furthermore, 
	\begin{align}
		\frac{\Phi_i  - a_i}{b_i} \Rightarrow \mathcal{N}(0,1), 	
	\end{align}
	where
	\begin{align*}
		&a_i = \frac{(d_i+1)\sqrt{N}\sqrt{d_i^2 - y}}{d_i^3}\E \Bigg[   \sum_{k=1}^Nx_{ik}^2 \sum_{u=1,u\neq i}^M\Big(x_{uk}^2 - \frac{1}{N}\Big)\Bigg],  \\
		&b_i^2 = 2(1+d_i^{-1})^2 + \frac{(d_i^2-y)(d_i+1)^2}{d_i^4} \Big(\E \big(\sqrt{N}x_{ik}\big)^4 - 3\Big). 
	\end{align*}
\end{thm}

\subsection{Proof Strategy} First, our proofs rely on recently developed poly-log bound for Cheeger constant of log concave measures. In the recent breakthrough \cite{chen} and \cite{KL22} on the KLS conjecture, a lower bound better than $n^{-\epsilon}$ with any $\epsilon>0$ is obtained for the Cheeger constant of log concave measure on $\mathbb{R}^n$. Specifically, a poly-log bound $\log^{-5} n$ is recently derived in \cite{KL22}. Based on this near constant lower bound and the fundamental concentration result of Gromov and Milman, we have a nearly optimal concentration for the Lipschitz function of the isotropic log concave random vectors; see Theorem \ref{Lip Theorem} and Corollary \ref{Lip corollary} below. With the aid of the Lipschitz concentration, following the argument in \cite{Ada}, one can then derive a Hanson-Wright type inequality for the quadratic forms of the isotropic log concave random vectors, with a nearly optimal tail probability bound; see Lemma \ref{Large deviation lemma} below. 

The Hanson-Wright type inequality, or the large deviation inequality for the quadratic forms of the columns of our data matrix $X$, is a key input for our random matrix analysis. From \cite{YinPi}, one knows that as long as a data matrix $X$ has i.i.d. columns and the columns satisfy sufficiently strong concentrations for their linear forms and quadratic forms, one will have the local MP law (cf. Theorem \ref{thm. local law}) and the rigidity of the eigenvalues ( cf. Theorem \ref{rigidity}). With these precise large deviation controls on the spectrum, we can then proceed with the analysis of the limiting distributions of the extreme eigenvalues. For edge universality, roughly speaking, we will follow the dynamic approach together with the Green function comparison approach developed by the Erd\H{o}s-Yau school. We refer to the monograph \cite{EY} for a detailed introduction of this framework. For our model, more specifically, we will first construct an interpolation between $X$ and the matrix $X^{\mathbf{w}}$ with i.i.d. $N(0,1/N)$ entries. That is $	X(t) := \sqrt{1-t} X + \sqrt{t}X^{\mathbf{w}}$, and $H(t)=X(t)X^*(t)$ is the corresponding covariance matrix. The work \cite{DY}, an adaptation of the work \cite{LY} from the Dyson Brownian Motion to the rectangular Dyson Brownian Motion, indicates that the edge universality holds for our $H(t)$ as long as $t\geq t_0=N^{-\frac13+\epsilon}$. Then what remains is to compare the distribution of the extreme eigenvalues of $H(t_0)$ with that of the original matrix $H=H(0)$. For this step, we adopt a Green function comparison approach which dates back to \cite{YinPi}. The Green function comparison heavily relies on effective controls of high order correlation of the components of the matrix columns. However, it is known that for the general log concave random vectors, estimating the high-order correlation is challenging. Therefore, instead of proving the edge universality of the most general log concave ensemble, we consider a sub class with an additional unconditional assumption. We then show that under the conditional assumption, the high order correlation of the vector components effectively resemble the i.i.d. counterparts. Finally, when the population covariance is not exactly isotropic, we also consider the case when spikes are present. Relying on the local law and the high order correlation control, we can then also derive the fluctuation of the extreme eigenvalues when the spikes are above the threshold.

\subsection{Notations and Conventions} \label{s.NC}

Throughout this paper, we regard $N$ as our fundamental large parameter. Any quantities that are not explicit constant or fixed may depend on $N$; we almost always omit the argument $N$ from our notation. We use $\|u\|_\alpha$ to denote the $\ell^\alpha$-norm of a vector $u$. We further use $\|A\|_{(\alpha,\beta)}$ to denote the induced norm $\sup_{x\in \mathbb{C}^n, \|x\|_\alpha=1} \|Ax\|_\beta$ for an $A\in \mathbb{C}^{m\times n}$. We write  $\|A\|\equiv \|A\|_{(2,2)}$ for the usual operator norm of a matrix $A$.  We use $C$ to denote some generic (large) positive constant. The notation $a\sim b$ means $C^{-1}b \leq |a| \leq Cb$ for some positive constant $C$. Similarly, we use $a\lesssim b$ to denote the relation $|a|\leq Cb$ for some positive constant $C$. When we write $a \ll b$ and $a \gg b$ for possibly $N$-dependent quantities $a\equiv a(N)$ and $b\equiv b(N)$, we mean $|a|/b \to 0$ and $|a|/b \to \infty$ when $N\to \infty$, respectively.

\subsection{Organization}	The paper is organized as follows.  Section \ref{Sec. Preliminaries} is devoted to the preliminaries, which is crucial for later discussion.  In Section \ref{Sec Strong local MP law}, we establish the local MP law which further implies the spectral rigidity in Theorem \ref{rigidity} . The proof  Theorem \ref{universality thm} is stated in Section \ref{Sec. comparison}, based on a Green function comparison.  Finally,  in Section \ref{Sec Proofs for Spiked covariance matrices} and \ref{Appendix Remaining estimation in Section 6}, we prove  Theorem \ref{Spiked model distribution}. The proofs of some additional technical estimates are relegated to the appendix.

\section{Preliminaries}\label{Sec. Preliminaries}
In this section we collect some necessary  notations and technical tools.  First, we have the following elementary result about stochastic domination.
\begin{lemma} \label{prop_prec} Let
	\begin{equation*}
	\mathsf{X}_i=(\mathsf{X}^{(N)}_i(u):  N \in \mathbb{N}, \ u \in \mathsf{U}^{(N)}), \   \mathsf{Y}_i=(\mathsf{Y}_i^{(N)}(u):  N \in \mathbb{N}, \ u \in \mathsf{U}^{(N)}),\quad i=1,2
	\end{equation*}
	be families of  random variables, where $\mathsf{Y}_i, i=1,2,$ are nonnegative, and $\mathsf{U}^{(N)}$ is a possibly $N$-dependent parameter set.	Let 
	\begin{align*}
	\Psi=(\Psi^{(N)}(u): N \in \mathbb{N}, \ u \in \mathsf{U}^{(N)})
	\end{align*}
	be a family of deterministic nonnegative quantities. We have the following results:
	
(i)	If $\mathsf{X}_1 \prec \mathsf{Y}_1$ and $\mathsf{X}_2 \prec \mathsf{Y}_2$ then $\mathsf{X}_1+\mathsf{X}_2 \prec \mathsf{Y}_1+\mathsf{Y}_2$ and  $\mathsf{X}_1 \mathsf{X}_2 \prec \mathsf{Y}_1 \mathsf{Y}_2$.

 (ii) Suppose $\mathsf{X}_1 \prec \Psi$, and there exists a constant $C>0$ such that  $|\mathsf{X}_1^{(N)}(u)| \leq N^{C}\Psi^{(N)}(u)$ a.s. uniformly in $u$ for all sufficiently large $N$. Then $\E \mathsf{X}_1 \prec \Psi$. 
\end{lemma}
The high probability event is then defined as follows.
\begin{defin}[High probability event] \label{def.high-probab}
We say an event $\mathcal{E}\equiv \mathcal{E}(N)$ holds with {\it high probability } (in $N$) if for any fixed $D>0$, $\mathbb{P}(\mathcal{E}^c)\leq N^{-D}$ when $N$ is sufficiently large. 
\end{defin}

Recall the matrix $Q$ from (\ref{def of Q}). Its Green function will be denoted by
\begin{align*}
	G(z) := (Q - zI)^{-1}, \quad z = E + \mathrm{i}\eta, \quad E \in \mathbb{R}, \eta > 0.
\end{align*}
In the sequel, we also need to consider $\mathcal{H}:= X^*X$ which shares the same non-zero eigenvalues with $H$. We further denote the Green functions of $H$ and $\mathcal{H}$ respectively by
\begin{align*}
	\mathcal{G}_1(z) = (H - zI)^{-1}, \quad \mathcal{G}_2(z) = (\mathcal{H} - zI)^{-1}, \quad z = E + \mathrm{i}\eta, \quad E \in \mathbb{R}, \eta > 0.
\end{align*} 
The corresponding normalized traces are defined by
\begin{align}\label{empirical stieljes}
	m_{1N}(z) := \frac{1}{M}\Tr \mathcal{G}_1(z) = \int (x-z)^{-1} \mathrm{d}F_{1N}(x),\quad 
	m_{2N}(z) := \frac{1}{N}\Tr \mathcal{G}_2(z) = \int (x-z)^{-1} \mathrm{d}F_{2N}(x),
\end{align}
where  $F_{1N}(x)$, $F_{2N}(x)$ are the empirical spectral distributions of $H$ and $\mathcal{H}$ respectively, i.e.,
\begin{align}
F_{1N}(x):=\frac{1}{M}\sum_{i=1}^M \mathds{1}(\lambda_i(H)\leq x), \quad F_{2N}(x):=\frac{1}{N}\sum_{i=1}^N \mathds{1}(\lambda_i(\mathcal{H})\leq x). \nonumber
\end{align}

When the entries of $X$ are i.i.d. with suitable moment condition, it is well-known since \cite{MP67} that $F_{1N}(x)$ and $F_{2N}(x)$ converge weakly (a.s.) to the MP laws 
$\nu_{y,1}$ and $\nu_{y,2}$ (respectively) given below
\begin{align}
&\nu_{y,1}({\rm d}x):=\frac{1}{2\pi xy}\sqrt{\big((\lambda_+-x)(x-\lambda_-)\big)_+}{\rm d}x+\Big(1-\frac{1}{y}\Big)_+\delta({\rm d}x),\nonumber\\
&\nu_{y,2}({\rm d}x):=\frac{1}{2\pi x}\sqrt{\big((\lambda_+-x)(x-\lambda_-)\big)_+}{\rm d}x+(1-y)_+\delta({\rm d}x), \label{19071801}
\end{align}
where $\lambda_{\pm}:=(1\pm \sqrt{y})^2$. Note that here the parameter $y$ may be $N$-dependent. Hence, the weak convergence (a.s.) shall be understood as $\int g(x) {\rm d} F_{aN}(x)-\int g(x) \nu_{y,a}({\rm d}x) \stackrel{a.s.} \longrightarrow 0 $ for any given bounded continuous function $g:\mathbb{R}\to \mathbb{R}$, for $a=1,2$. We further denote by $F_a$ the cumulative distribution function of $\nu_{y,a}$ for $a=1,2$. 
Note that $m_{1N}$ and $m_{2N}$ can be regarded as the Stieltjes transforms of $F_{1N}$ and $F_{2N}$, respectively.  We further define their deterministic counterparts, i.e.,  Stieltjes transforms of $\nu_{y,1},\nu_{y,2}$,  by $m_1(z),m_2(z)$,  respectively, i.e.,
 \begin{align}
 m_a(z):=\int (x-z)^{-1}\nu_{y,a}({\rm d}x), \quad a=1,2. \nonumber
 \end{align}
 From the definition (\ref{19071801}), it is elementary to compute 
 \begin{align}
m_1(z)=\frac{1-y-z+\mathrm{i} \sqrt{(\lambda_+-z)(z-\lambda_-)}}{2zy},  \quad m_2(z)=\frac{y-1-z+\mathrm{i} \sqrt{(\lambda_+-z)(z-\lambda_-)}}{2z}, \label{m1m2}
 \end{align}
where the square root is taken with a branch cut on the negative real axis. Equivalently, we can also characterize $m_1(z),m_2(z)$ as the unique solutions from $\mathbb{C}^+$ to $\mathbb{C}^+$ to the equations
\begin{align}
zym_1^2+[z-(1-y)]m_1+1=0, \qquad  zm_2^2+[z+(1-y)]m_2+1=0. \label{selfconeqt}
\end{align} 
Using \eqref{m1m2} and \eqref{selfconeqt}, one can easily derive the following identities  
\begin{align}
m_1=-\frac{1}{z(1+m_2)}, \quad 1+zm_1=\frac{1+zm_2}{y},\quad m_1\big((zm_2)'+1\big)=\frac{m_1'}{m_1}, \label{identitym1m2}
\end{align}
which will be used in the later discussions.

We will also often consider the minors  of $X$ which are denoted as follows. For $\mathbb{T} \subset \{1,\cdots N \}$ we denote by $X^{(\mathbb{T})}$ the $(M \times (N - |\mathbb{T}|))$ minor of $X$ obtained by removing all columns of $X$ indexed by $i \in \mathbb{T}$. The quantities $ \mathcal{G}_1^{(\mathbb{T})}(z), \mathcal{G}_2^{ (\mathbb{T})}(z), \lambda_\alpha^{(\mathbb{T})}$ are defined similarly using $X^{(\mathbb{T})}$. Analogously, 
	For $\mathbb{T} \subset \{1,\cdots M \}$ we define $X^{[\mathbb{T}]}$ as the $((M - |\mathbb{T}|) \times N )$ minor of $X$ obtained by removing all rows of $X$ indexed by $j \in \mathbb{T}$. The quantities $ \mathcal{G}_1^{[\mathbb{T}]}(z), \mathcal{G}_2^{ [\mathbb{T}]}(z), \lambda_\alpha^{[\mathbb{T}]}$ are defined similarly using $X^{[\mathbb{T}]}$.
	
In the sequel, we collection some basic properties of the log-concave measures. The following proposition can be found in \cite{well}, for instance.
\begin{prop}\label{iso prop}
	Suppose that $x,y$ are independent isotropic log-concave random vectors in $\mathbb{R}^n$, we have the following.
	\begin{itemize}
		\item[(\it{i})] (Preservation by convolution) $\sqrt{t}x+\sqrt{1-t}y$ is also isotropic log-concave  for any $t\in [0,1]$.  
		\item [(\it{ii})](Preservation by marginalization) Any sub-vector of $x$ is isotropic log-concave.
	\end{itemize}
\end{prop}
Further it is also well-known that   log-concave measures have sub-exponential tails, see \cite{cule2010theoretical}. Hence, all moments of its components exist. We  then recall the following fundamental result of  Gromov and Milman \cite{GromovMilman}; see also \cite{LV}.
\begin{thm}[Lipschitz concentration]\label{Lip Theorem}
	 For any $L$-Lipschitz function $g:\mathbb{R}^n\to \mathbb{R}$, and isotropic log-concave  $\mathbf{y}\in \mathbb{R}^n$ with density  $\rho$,
	 \begin{align}\label{Lip concentration proba}
	 	\mathbb{P} (|g(\mathbf{y}) - \E g(\mathbf{y})| > cLt) \le \mathrm{e}^{-t\psi_\rho},
	 \end{align}
	 where $\psi_\rho$ is the Cheeger constant of $\rho$ defined by
	 \begin{align*}
	 	\psi_\rho := \inf_{A \in \mathbb{R}^n} \frac{\int_{\partial A} \rho(x) \mathrm{d}x}{\min \left\{ \int_{A} \rho(x)\mathrm{d}x ,\int_{\mathbb{R}^n \backslash A} \rho(x)\mathrm{d}x\right\}}. 
	 \end{align*}
\end{thm}

The optimal lower bound for $\psi_\rho$ is a long standing problem in the field of high-dimensional convex geometry. It is conjectured that the lower bound for $\psi_\rho$ should be dimensional-free. 
In the very recent breakthrough of this problem \cite{chen, KL22}, this bound is improved to of order $\log^{-5} n$.  
Therefore, using this estimate of the Cheeger constant gives a bound of $\mathrm{e}^{-ct/\log^5 n}$ at the RHS of (\ref{Lip concentration proba}) for any isotropic log-concave density $\rho$ in $\mathbb{R}^n$. The result can be summarised as follows,
\begin{corollary}\label{Lip corollary}
	 For any $L$-Lipschitz function $g:\mathbb{R}^n\to \mathbb{R}$, and isotropic log-concave  $\mathbf{y}\in \mathbb{R}^n$,
	 \begin{align}
	 	\mathbb{P} (|g(\mathbf{y}) - \E g(\mathbf{y})| > cLt) \le \mathrm{e}^{-ct/(\log n)^5}.
	 \end{align}
\label{Lip concentration}
\end{corollary} 
\begin{remark}
	We remark here that although the KLS conjecture is still open, the current bound $\psi_\rho \gtrsim \log^{-5} n$ suffices for our purpose in this paper.
\end{remark}

\section{Local MP law}\label{Sec Strong local MP law}
Our goal in this section is to establish the  local MP law. The following is the main result in this section, which is also our main technical input for establishing other main theorems. Set the domain 
	\begin{align}\label{spectral domain}
			D(\epsilon):=\{ z \in \mathbb{C}: \mathbf{1}_{y<1}(\lambda_-/5) \le E \le 5\lambda_{+}, N^{-1+\epsilon}\le \eta \le 10(1+y) \}.
		\end{align}
\begin{thm}[Local MP law] \label{thm. local law}
	Suppose that Assumption \ref{assu.1.1} holds. Let $\epsilon > 0$ be given. We have
	\begin{itemize}
		\item[(i)] The Stieltjes transform of the empirical spectral distribution of $H$ satisfies
		\begin{align}\label{trace gap}
		 |m_{2N}(z) - m_{2}(z)| \prec   \frac{1}{N\eta},
		\end{align}
		uniformly on $D(\epsilon)$.  
			
		\item[(ii)] The individual matrix elements of the Green function satisfy
		\begin{align}\label{entry wise local law}
			 |[\mathcal{G}_{2}(z)]_{ij} - m_{2}(z)\delta_{ij}| \prec  \sqrt{\frac{\Im m_{2}(z)}{N\eta} } + \frac{1}{N\eta}
		\end{align}
		uniformly on $D(\epsilon)$. 
		\end{itemize}
\label{Strong MP law}\end{thm}

\begin{remark}
 	We remark here (\ref{trace gap}) can be improved when $\lambda_+ \le  E \le 5\lambda_+$. Similar to the analysis in \cite{YinPi}, the following stronger estimate holds for $\lambda_+ \le  E \le 5\lambda_+$,
\begin{align}\label{eta star regular 2}
	|m_{2N}(z) - m_2(z)| \prec \frac{1}{N(\kappa +  \eta)} + \frac{1}{(N\eta)^2\sqrt{\kappa + \eta}},
\end{align}
where $\kappa = |E - \lambda_+|$ is the distance to the rightmost edge.
\end{remark}

Based on Theorem 3.6 in \cite{YinPi} and the analysis therein, we know that to prove Theorems \ref{Strong MP law} and also \ref{rigidity}, it suffices to prove the following large deviation lemma. 
\begin{lemma}[Large deviation lemma]\label{Large deviation lemma}
	Suppose that Assumption \ref{assu.1.1} holds. For any fixed $k \le N$, and for any deterministic $A_i \in \mathbb{C}$ and $B_{ij} \in \mathbb{C}$, we have
	\begin{align}
		&\bigg| \sum_{i=1}^M x_{ik} A_i \bigg| \prec \frac{1}{\sqrt{N}}\bigg( \sum_{i}|A_{i}|^2 \bigg)^{1/2} ,\label{Linear form} \\
		&\bigg| \sum_{i,j=1}^M x_{ik} B_{ij} x_{jk} -  \frac{1}{N}\sum_{i=1}^M B_{ii} \bigg| \prec \frac{1}{N}\bigg( \sum_{i,j}|B_{ij}|^2 \bigg)^{1/2}.\label{Quadratic form}
	\end{align}
\end{lemma}

\begin{proof}[Proof of Theorems \ref{Strong MP law} and  \ref{rigidity}] By Theorem 3.6 in \cite{YinPi}, with the aid of Lemma \ref{Large deviation lemma}, one can conclude Theorems \ref{Strong MP law} and  \ref{rigidity}. 
\end{proof}

Hence, what remains is the proof of Lemma \ref{Large deviation lemma}. 
\begin{proof}[Proof of Lemma \ref{Large deviation lemma}]
	For any isotropic log-concave random vector $q \in \mathbb{R}^{M}$ , let $g(q) := \sum_{i=1}^M q_{i}A_i$. Notice that the Lipschitz constant of $g$ is apparently $\|A\|_2=\sqrt{\sum_{i}|A_{i}|^2 }$.  Then (\ref{Linear form}) follows from Corollary \ref{Lip concentration} directly. 
The large deviation for quadratic form (\ref{Quadratic form}) follows from a Hanson-Wright type inequality for the log-concave random vectors which can be derived using the arguments in \cite{Ada}  based on Corollary \ref{Lip concentration}, with only slight modification. As pointed out in \cite{Ada} (Remark 2.8),  one can modify accordingly the tail probability of the resulting Hanson-Wright type inequality under the different assumptions on rate of decay for the tail probability of the Lipschitz concentration.  To be more precisely, we can show that for any $B = (B_{ij})_{i,j=1}^{M}$ and for every $t > 0$, there exists a constant $C > 0$ such that
	\begin{align}
		\mathbb{P}\left( \left| \sum_{i,j=1}^M x_{ik} B_{ij} x_{jk} -  \frac{1}{N}\sum_{i=1}^M B_{ii} \right|  \ge t \right) \le 2 \exp \left( -\frac{\psi_\rho}{C} \min \left( \frac{Nt}{\|B \|_{HS}}, \left( \frac{Nt}{\| B\|} \right)^{1/2} \right) \right). \label{121701}
	\end{align}
	For the convenience of the reader, in Appendix we briefly reproduce Adamczak's argument to prove (\ref{121701}), under a different tail probability of the Lipschitz concentration. 
Considering the current best bound $\psi_\rho \gtrsim \log^{-5} M$ with $t = N^{\epsilon-1}\|B \|_{HS}$ for arbitrary small $\epsilon >0$, we can conclude the proof of (\ref{Quadratic form}). 
\end{proof}

\section{Proof of Theorem \ref{universality thm}}\label{Sec. comparison}

In this section, we prove Theorem \ref{universality thm},  with the unconditional assumption (cf. Definition \ref{def.unconditional}). 
For any $t \in [0,1]$, we define
\begin{align}\label{OU process}
	X(t) := \sqrt{1-t} X + \sqrt{t}X^{\mathbf{w}},
\end{align}
where $X(0) = X$ and $X^{\mathbf{w}}$ is an $M \times N$ random matrix with i.i.d $N(0,1/N)$ Gaussian entries. We then set $H(t)=X(t)X(t)^*$. First, the edge universality of Dyson Brownian motion by Landon-Yau in \cite{LY} indicates that when the spectrum of the initial $H$ is sufficiently regular, then after $t\geq N^{-\frac13+\varepsilon}$, the local equilibrium at the edge can be achieved. The covariance matrix adaptation of this dynamical result of edge universality was derived in \cite{DY}, and the regularity of the initial spectrum is guaranteed by the spectral rigidity. Hence, we will have the edge universality for $H(t)$ with $t_0=N^{-\frac13+\epsilon}$. Then  Green function comparison between $H\equiv H(0)$ and $H(t_0)$ allows us to prove the edge universality in Theorem \ref{universality thm}. In the sequel, we show the detailed discussion.

First, notice that the random matrix $X(t)$ is always unconditional isotropic log-concave (cf. Proposition \ref{iso prop}) for all $t\in[0,1]$. Our argument can be divided into two steps, 
\begin{itemize}
	\item[(\textit{i})] Edge Universality for $X(t)X(t)^*$ for $t\geq t_0:=N^{-\frac13+\epsilon}$.
	\item[(\textit{ii})] Comparison between the extreme eigenvalue distributions of $X(t_0)X(t_0)^*$ and $XX^*$.
\end{itemize}

Having established local MP law for log-concave ensembles, which ensures the square root behavior of the spectral density at the edge, the convergence analysis for the extreme eigenvalue of matrix ensemble $X(t)$ is identical to the analysis in \cite{DY} (Proof of Theorem III.2). For the convenience of the reader, we will sketch the proof of step (\textit{i}) in Section \ref{Proof of Universality Step 1}. 

The main focus of the proof  lies in step (\textit{ii}). The basic framework to prove step (\textit{ii}) is an adaptation of the Green function comparison in \cite{YinPi, YinPib}, which has been used to prove universality results of sample covariance/correlation matrices and many other different random matrix models. However, previous comparison techniques used for sample covariance/correlation matrix \cite{YinPi} were hindered by the high-order correlation structure of the log-concave ensembles. To circumvent this, we give an estimation on the mixed moments of the coordinates of the unconditional isotropic log-concave random vector.  The detailed proof is given in Section \ref{Proof of Universality Step 2}.

\subsection{Edge Universality of $X(t_0)X(t_0)^*$}\label{Proof of Universality Step 1}
The regularity conditions for the square root behavior near the spectral edge is stated as the following. 
\begin{defin}[$\eta_*$-regular]
	Let $\eta_*$ be a deterministic parameter satisfying $\eta_* := N^{-\phi_*}$ for some constant $0 < \phi_* \le 2/3$. We say an $M \times M$ matrix $V$ is $\eta_*$-regular around the right edge $\lambda_+^V := \lambda_1(V)$ if the following properties hold for some constants $c_V$, $C_V > 0$.
	\begin{itemize}
		\item[(i)] For $z = E + \mathrm{i}\eta$ with $\lambda_+^V - c_V \le E \le \lambda_+^V$ and $\eta_* + \sqrt{\eta_*|\lambda_+^V-E|} \le \eta \le 10$, we have
		\begin{align}
			\frac{1}{C_V}\sqrt{|\lambda_+^V-E|+\eta} \le \Im m_V(E+\mathrm{i}\eta) \le C_V\sqrt{|\lambda_+^V-E|+\eta}.
		\end{align}
		\item[(ii)] For $z = E + \mathrm{i}\eta$ with $\lambda_+^V \le E \le \lambda_+^V + c_V $ and $\eta_*  \le \eta \le 10$, we have
		\begin{align}
			\frac{1}{C_V}\frac{\eta}{\sqrt{|\lambda_+^V-E|+\eta}} \le \Im m_V(E+\mathrm{i}\eta) \le C_V \frac{\eta}{\sqrt{|\lambda_+^V-E|+\eta}}.
		\end{align}
		\item[(iii)] We have $2c_V \le \lambda_+^V \le C_V/2$.
	\end{itemize}
	Here $m_V(z) := M^{-1} \sum_{i=1}^M(\lambda_i(V)-z)^{-1}, z \in \mathbb{C}^{+}$ denotes the Stieljes transform of $V$.
\end{defin}
Recall $t_0 = N^{-1/3+\epsilon}$ for some small constant $\epsilon > 0$. 
We denote $\tilde{X} \equiv X(t_0)$ for simplicity. The corresponding sample covariance matrix is defined as $\tilde{H}:= \tilde{X} \tilde{X} ^*$. Note that $X(t)$ from (\ref{OU process}) can be regarded as a rectangular matrix Dyson Brownian Motion starting from $X$. It is routine to prove that the $\eta^*$-regularity holds with high probability for our initial matrix $H$, with the aid of the local MP law and rigidity in Theorems \ref{thm. local law} and \ref{rigidity}.

With the $\eta^*$-regularity, by Theorem 4 in \cite{DY}, we can conclude that there exists a positive parameter $\gamma_0$ of order $1$ such that for any $x \in \mathbb{R}$,
\begin{align}\label{intermediate result 1}
	\lim_{N \to \infty} \mathbb{P}\left( \gamma_0N^{2/3}(\lambda_1(\tilde{H}) - \lambda_{+,t_0} ) \le x \right) = \lim_{N \to \infty} \mathbb{P}\left( N^{2/3}(\lambda_1(\text{GOE}) - 2 ) \le x \right), 
\end{align}
for all $x \in \mathbb{R}$, where we denote $\lambda_1(\text{GOE})$ as the largest eigenvalue of the Gaussian Orthogonal Ensemble. Here we clarify that $\lambda_{+,t_0}$ is the right edge of $\rho_{H,t_0}$, which is the rectangular free convolution of $\rho_{0} := M^{-1}\sum_{i=1}^M\delta_{\lambda_i(H)}$ with MP law at time $t_0$. Following the analysis in \cite{DY} (Appendix A, proof of Theorem III.2), one can further replace $\lambda_{+,t_0}$ with $\lambda_+$ in (\ref{intermediate result 1}) by showing that
\begin{align*}
	N^{2/3}|\lambda_{+,t_0} - \lambda_+| \overset{\mathbb{P}}{\rightarrow} 0.
\end{align*}
We omit the details. Further notice that $\gamma_0$ can be determined by the square root behavior of MP law near the edge, with the well-known result of Tracy-Widom limit of Wishart matrix, we can obtain for any $x \in \mathbb{R}$,
\begin{align}\label{result 1}
	\lim_{N \to \infty} \mathbb{P}^{\tilde{H}}\left( N^{2/3}(\lambda_1(\tilde{H}) - \lambda_{+} ) \le x \right) = \lim_{N \to \infty} \mathbb{P}^{W}\left( N^{2/3}(\lambda_1(W) - \lambda_{+} ) \le x \right)
\end{align}
This completes the proof for step (\textit{i}).

\subsection{Green function comparison between  $\tilde{X}$ and $X$}\label{Proof of Universality Step 2}
Recall $\tilde{X} \equiv X(t_0)$ with $t_0 = N^{-1/3+\epsilon}$. Notice that the isotropic log-concavity property is preserved when running the Brownian Motion (\ref{OU process}). Hence, Theorems \ref{rigidity} and  \ref{thm. local law} still hold for $\tilde{X}$. Our goal in this section is to prove the following proposition. 
\begin{prop} \label{prop.universality} Under the assumption of Theorem \ref{universality thm}, we have 
\begin{align}\label{result 2}
	\lim_{N \to \infty} \mathbb{P}^{\tilde{H}}\left( N^{2/3}(\lambda_1(\tilde{H}) - \lambda_{+} ) \le x \right) = \lim_{N \to \infty} \mathbb{P}^{H}\left(N^{2/3}(\lambda_1(H) - \lambda_{+} ) \le x  \right).
\end{align}
\end{prop}

\begin{proof}[Proof of Theorem \ref{universality thm}]
Combining (\ref{result 2}) and (\ref{result 1}), we arrive at
\begin{align*}
	\lim_{N \to \infty} \mathbb{P}^{H}\left( N^{2/3}(\lambda_1(H) - \lambda_{+} ) \le x \right) = \lim_{N \to \infty} \mathbb{P}^{W}\left(N^{2/3}(\lambda_1(W) - \lambda_{+} ) \le x  \right).
\end{align*}
This completes the proof of Theorem \ref{universality thm}.  
\end{proof}

As shown in \cite{YinPi}, to prove (\ref{result 2}), it suffices to establish a Green function comparisons for $\tilde{H}$ and $H$ on the edge, i.e., Proposition \ref{Green function comparison thm}.  Let $\tilde{\mathcal{G}}_{2}(z) : =  (\tilde{\mathcal{H}} - zI)^{-1}$ be the Green function of $\tilde{\mathcal{H}}$ with $\tilde{\mathcal{H}} := \tilde{X}^*\tilde{X}$, and the corresponding Stieltjes transform is denoted by $\tilde{m}_{2N}(z) := N^{-1} \Tr \tilde{\mathcal{G}}_{2}(z)$. 
\begin{prop}[Green function comparison] \label{Green function comparison thm}
	Let $F: \mathbb{R} \to \mathbb{R}$ be a function whose derivatives satisfy
	\begin{align*}
		\max_x |F^{\alpha}(x)|(|x|+1)^{-C_1} \le C_1, \quad  \alpha = 1,2,3,4
	\end{align*}
	for some constant $C_1 > 0$. Then there exist $\epsilon_0 > 0$, $N_0 \in \mathbb{N}$ and $\delta > 0$ depending only on $C_1$ such that for any $\epsilon < \epsilon_0$, $N \ge N_0$ and real numbers $E, E_1$ and $E_2$ satisfying 
	\begin{align*}
		|E - \lambda_+| \le N^{-2/3+\epsilon}, \quad |E_1 - \lambda_+| \le N^{-2/3+\epsilon}, \quad |E_2 - \lambda_+| \le N^{-2/3+\epsilon},
	\end{align*}
	and $\eta_0 = N^{-2/3-\epsilon}$, we have
	\begin{align}\label{comparison 1}
		|\E F(N\eta_0\Im \tilde{m}_{2N}(z)) - \E F(N\eta_0\Im m_{2N}(z)) | \le C N^{-\delta+C\epsilon}, \quad z = E + \mathrm{i}\eta_0,
	\end{align}
	and
	\begin{align}\label{comparison 2}
		\left| \E F\left( N\int_{E_1}^{E_2}\mathrm{d}y \Im \tilde{m}_{2N}(y+\mathrm{i}\eta_0) \right)- \E F\left( N\int_{E_1}^{E_2}\mathrm{d}y \Im m_{2N}(y+\mathrm{i}\eta_0) \right) \right|\le C N^{-\delta+C\epsilon}
	\end{align}
	for some constant $C$.
\end{prop}

\begin{proof}[Proof of Proposition \ref{prop.universality}] Based on Proposition \ref{Green function comparison thm}, the proof of Proposition \ref{prop.universality} is  the same as the proof of Theorem 1.1. in \cite{YinPi}, and thus we omit the details here. 
\end{proof}

The rest of this section is devoted to the proof of Proposition \ref{Green function comparison thm}. 

\begin{proof}[Proof of Proposition \ref{Green function comparison thm}] In the sequel, we show the proof of (\ref{comparison 1}) only. The proof of (\ref{comparison 2}) is similar, and thus we omit it.  Similarly to the strategy in \cite{YinPi}, we replace the columns of $X$ by those of $\tilde{X}$ one by one and estimate the one-step difference between the functionals of two Green functions. 
For $1 \le  \gamma \le N$, denote by $X_{\gamma}$ the random matrix whose $j$-th column is the same as that of $\tilde{X}$ if $j \le  \gamma$ and that of $X$ otherwise. In particular, $X_N = \tilde{X}$ and $X_0 = X$. Set
\begin{align*}
	\mathcal{H}_\gamma := X_\gamma^* X_\gamma, \quad \mathcal{G}_{2,\gamma}(z) := (\mathcal{H}_{\gamma} - zI)^{-1}, \quad m_{2N,\gamma}(z) := \frac{1}{N}\Tr \mathcal{G}_{2,\gamma}(z). 
\end{align*}
Then we have telescoping sum
\begin{align}\label{telescoping sum}
	\E F(N\eta_0\Im \tilde{m}_{2N}(z)) - \E F(N\eta_0\Im m_{2N}(z)) = \sum_{\gamma = 1}^N \E F(N\eta_0\Im m_{2N,\gamma}(z)) - \E F(N\eta_0\Im m_{2N,\gamma-1}(z)).
\end{align}
Recall the notations of minors in Section \ref{Sec. Preliminaries}. Define
\begin{align*}
	\mu_\gamma = \eta_0 \Im \Tr \mathcal{G}_{2,\gamma}^{(\gamma)}(z) - \Im \frac{\eta_0}{z}.
\end{align*}
We can further rewrite  (\ref{telescoping sum}) as 
\begin{align}\label{telescoping sum 2}
	&\E F(N\eta_0\Im \tilde{m}_{2N}(z)) - \E F(N\eta_0\Im m_{2N}(z)) \notag\\
	=&  \sum_{\gamma = 1}^N (\E F(\eta_0\Im \Tr \mathcal{G}_{2,\gamma}(z)) - \E F(\mu_\gamma)) - (\E F(\eta_0\Im \Tr \mathcal{G}_{2,\gamma-1}(z)) - \E F(\mu_\gamma)).
\end{align}
We can then complete the proof of Proposition \ref{Green function comparison thm} by establishing the following lemma. 
\begin{lemma}\label{telescoping error lemma}
	Under the same condition as Proposition \ref{Green function comparison thm}, for any $\gamma \in \{1,\cdots,N \}$, we have 
	\begin{align}\label{telescoping error}
		|\E F(\eta_0\Im \Tr \mathcal{G}_{2,\gamma}(z)) - \E F(\mu_\gamma)|,\; |\E F(\eta_0\Im \Tr \mathcal{G}_{2,\gamma-1}(z)) - \E F(\mu_\gamma)| \lesssim N^{-7/6+\epsilon}.
	\end{align}
\end{lemma}
The proof of Lemma \ref{telescoping error lemma} is postponed to the end of this section. By Lemma \ref{telescoping error lemma}, the telescoping sum (\ref{telescoping sum 2}) can be bounded as follows
\begin{align*}
	|\E F(N\eta_0\Im \tilde{m}_{2N}(z)) - \E F(N\eta_0\Im m_{2N}(z))| \lesssim N^{-1/6 + \epsilon},
\end{align*}
which implies (\ref{comparison 1}). Similarly, one can show (\ref{comparison 2}), which completes the proof of Proposition \ref{Green function comparison thm}.

\end{proof}

We first present a key lemma which will be used frequently in the proof of Lemma \ref{telescoping error lemma}.
\begin{lemma}\label{approximation lemma}
	Assume $\mathbf{y}=(y_1, \ldots, y_M) \in \mathbb{R}^{M}$ is an unconditional isotropic log-concave random vector scaled by $1/\sqrt{N}$. Let $\ell$ be any fixed integer. 
	For any fixed $k_1, \cdots, k_\ell$, define
\begin{align}
	B^{\pm}_{k_1,\cdots,k_\ell} := \Big\{ k_{\ell+1}\in \{1, \ldots, M\}: \pm \Big(\E \prod_{i=1}^{\ell+1} y_{k_i}^2 - N^{-1}\E \prod_{i=1}^\ell  y_{k_i}^2\Big) \ge N^{-(3\ell+4)/3} \Big\}, \label{B+k1k2}
\end{align}
We have 
$
	\# B^{\pm}_{k_1,\cdots,k_\ell} \lesssim N^{2/3}
$.
\end{lemma}
\begin{proof}[Proof of Lemma \ref{approximation lemma}]
We shall use the following thin shell bound for the isotropic log-concave random vectors with unconditional basis, which is proved in \cite{UTC}.
\begin{lemma}\label{thin shell}
	Assume $\mathbf{y}=(y_1, \ldots, y_n) \in \mathbb{R}^{n}$ is an isotropic log-concave random vector with unconditional basis, scaled by $1/\sqrt{N}$. we have
	\begin{align*}
		\text{Var}(\| \mathbf{y}\|_2^2)  \le C\frac{n}{N^2},
	\end{align*}
	for a  universal  positive constant $C$.
\end{lemma}
We show the estimate of $ \# B^{+}_{k_1,\cdots,k_\ell} $ only, and  $ \# B^{-}_{k_1,\cdots,k_\ell} $ can be handled similarly.
By direct calculation, we have
\begin{align}
	\E \sum_{k_{\ell+1} \in B^{+}_{k_1,\cdots,k_\ell}} \prod_{i=1}^{\ell+1} y_{k_i}^2=& \E  \prod_{i=1}^{\ell} y_{k_i}^2 \Bigg(\sum_{k_{\ell+1} \in B^{+}_{k_1,\cdots,k_\ell}}y_{k_{\ell+1}}^2 - \frac{\# B^+_{k_1,\cdots,k_{\ell}}}{N}\Bigg) + \frac{\# B^+_{k_1,\cdots,k_{\ell}}}{N} \E  \prod_{i=1}^{\ell} y_{k_i}^2  \notag \\
	\le & C\frac{\sqrt{\# B^+_{k_1,\cdots,k_{\ell}}}}{N^{\ell+1}} + \frac{\# B^+_{k_1,\cdots,k_{\ell}}}{N}\E \prod_{i=1}^{\ell} y_{k_i}^2 , \label{121401}
\end{align}
for some universal constant $C >0$. Here in the last step, we used the thin shell bound (cf. Lemma \ref{thin shell}) and a simple Cauchy-Schwarz. On the other hand, by the definition of $B^{+}_{k_1,\cdots,k_\ell}$, we have
\begin{align}
	\E \sum_{k_{\ell+1} \in B^{+}_{k_1,\cdots,k_\ell}} \prod_{i=1}^{\ell+1} y_{k_i}^2 \ge \frac{\# B^{+}_{k_1,\cdots,k_\ell}}{N} \E  \prod_{i=1}^{\ell} y_{k_i}^2  + \frac{\# B^{+}_{k_1,\cdots,k_\ell}}{N^{(3\ell+4)/3}}. \label{121402}
\end{align}
Combining (\ref{121401}) and (\ref{121402}), we must have
\begin{align*}
 \frac{\# B^{+}_{k_1,\cdots,k_\ell}}{N^{(3\ell+4)/3}} \le C\frac{\sqrt{\# B^+_{k_1,\cdots,k_{\ell}}}}{N^{\ell+1}}. 
\end{align*}
Then the bound on $\# B^{+}_{k_1,\cdots,k_\ell}$ follows. The bound on   $\# B^{-}_{k_1,\cdots,k_\ell}$ can be proved in the same way. 
\end{proof}

  We then proceed with  the proof of Lemma \ref{telescoping error lemma} based on Lemma \ref{approximation lemma}.
\begin{proof}[Proof of Lemma \ref{telescoping error lemma} ]We show the proof for $\gamma = 1$, and the other cases are identical. 

First, we claim the following bounds for $\mathcal{G}^{(1)}_1(z)$ with $z = E + \mathrm{i}\eta_0$, where $E$ and $\eta_0$ satisfy the assumption in Proposition \ref{Green function comparison thm}:
\begin{align}
	&|\langle x_1, (\mathcal{G}^{(1)}_1)^2x_1 \rangle| \prec N^{1/3+\epsilon},  \label{GQBound}\\
	&|[\mathcal{G}^{(1)}_1]_{ii} -  m_1| \prec N^{-1/3+\epsilon}, \; |[\mathcal{G}^{(1)}]_{ij}| \prec N^{-1/3+\epsilon }, \; i\neq j \label{GijBound}\\
	&|[ (\mathcal{G}^{(1)}_1 )^2 ]_{ij}| \prec  N^{1/3+\epsilon}, \label{G2ijBound}
\end{align}
The proofs of these bounds are postponed to Appendix \ref{sec Proofs of GQBound}. Hereafter we drop variable $z$ when there is no confusion.

 With the bounds (\ref{GQBound})-(\ref{G2ijBound}), one can easily show the following expansion
\begin{align}\label{expansion1}
	&\E F\left(\eta_0 \Im  \Tr \mathcal{G}_{2,0} \right) -\E F(\mu_1) = \E F^{(1)}(\mu_1) \eta_0 \Im \Big( \sum_{0 \le n < k \le 3} C_{k,n}\mathcal{Y}\mathcal{Z}^n \Big) + \E F^{(2)}(\mu_1) \eta_0^2\Big(\frac{1}{2}(\Im (C_{1,0}\mathcal{Y}))^2 \notag\\
	&+ \Im (C_{1,0}\mathcal{Y}) \Im (C_{2,0}\mathcal{Y}) + \Im (C_{1,0}\mathcal{Y})\Im (C_{2,1}\mathcal{Y}\mathcal{Z}) \Big) 
	+  \E F^{(3)}(\mu_1) \eta_0^3\Big(\frac{1}{6}(\Im (C_{1,0}\mathcal{Y}))^3 \Big) + O_\prec(N^{-4/3+\epsilon}),
\end{align}
where $C_{k,n} = O(1)$, $\mathcal{Y} = \langle x_1, (\mathcal{G}^{(1)}_1)^2x_1 \rangle$ and $\mathcal{Z} = \langle x_1, \mathcal{G}_1^{(1)}x_1 \rangle$. Define $\tilde{\mathcal{Y}} = \langle \tilde{x}_1, (\mathcal{G}_1^{(1)})^2\tilde{x}_1 \rangle$ and $\tilde{\mathcal{Z}} = \langle \tilde{x}_1, \mathcal{G}^{(1)}_1\tilde{x}_1 \rangle$ with $\tilde{x}_1$ be the first column of $\tilde{X}$, we can obtain that (\ref{expansion1}) also holds for the case when $\mathcal{Y}$ and $\mathcal{Z}$ are replaced by $\tilde{\mathcal{Y}}$ and $\tilde{\mathcal{Z}}$, respectively.

 By (\ref{expansion1}) and its tilde counterpart, we see that in order to show (\ref{telescoping error}), it suffices to show that for any $1 \le a \le 3$, $1 \le a+b \le 3$,
\begin{align}\label{Expectation bound}
	&\E \bigg( \eta_0^a \prod_{i=1}^a \langle x_1,Y_ix_1 \rangle \prod_{j=1}^b \langle x_1,Z_ix_1 \rangle \bigg)-\E \bigg( \eta_0^a \prod_{i=1}^a \langle \tilde{x}_1,Y_i\tilde{x}_1 \rangle \prod_{j=1}^b \langle \tilde{x}_1,Z_i\tilde{x}_1 \rangle \bigg) = O_\prec(N^{-1-\delta}),
\end{align} 
where $Y_i = (\mathcal{G}^{(1)}_1)^2$ or $((\mathcal{G}^{(1)}_1)^2)^*$ and $Z_j = \mathcal{G}^{(1)}_1$ or $(\mathcal{G}^{(1)}_1)^*$. 
We only consider the case when $Y_i = (\mathcal{G}_1^{(1)})^2$ and $Z_i = \mathcal{G}_1^{(1)}$, and the other cases are similar.  In the sequel, for convenience, for any integer $\ell\geq 1$,  we simply denote by $\sum_{k_1, \cdots, k_\ell}$ the sum over the index set $\{1,\ldots, M\}^\ell$ and we use $\sum_{k_1, \cdots, k_\ell}^{*}$ to denote the sum over $\{(k_1,\cdots, k_\ell)\in \{1,\ldots, M\}^\ell: k_i's \text{ are distinct}\}$. 
 
 Let $x_{k1}$ and $\tilde{x}_{k1}$ be the $k$-th element of $x_1$ and $\tilde{x}_1$, respectively.  First, when  $a = 1$ and $b = 0$, since both $x_1$ and $\tilde{x}_1$ are independent of $\mathcal{G}_1^{(1)}$, we have by the isotropic condition that
\begin{align*}
	\E  \eta_0  \langle x_1,(\mathcal{G}^{(1)}_1)^2x_1 \rangle  - \E  \eta_0  \langle \tilde{x}_1,(\mathcal{G}^{(1)}_1)^2\tilde{x}_1 \rangle   = 0. 
\end{align*}
Recall that $\eta_0 = N^{-2/3-\epsilon}$ and $t_0 = N^{-1/3+\epsilon}$. For $a = 3$, $b = 0$, using (\ref{G2ijBound}), we have the following crude bound, 
\begin{align*}
	\E \left( \eta_0^3  \langle x_1,(\mathcal{G}^{(1)}_1)^2x_1 \rangle^3 \right) - \E \left( \eta_0^3  \langle \tilde{x}_1,(\mathcal{G}^{(1)}_1)^2\tilde{x}_1 \rangle^3  \right) \lesssim t_0 \eta_0^3  \E\big(\max_{k_1, k_2} \big|[(\mathcal{G}^{(1)}_1)^2]_{k_1k_2}\big|\big)^3 = O_\prec(N^{-4/3+\epsilon}).
\end{align*} 

For the other cases, we need to estimate in more delicate way. We will only show the proof of the case $a = 1$ and $b=2$ in detail, since the others can be proved similarly. For brevity, we use the notation $\mathcal{A}(ab,ij,st) := \E [(\mathcal{G}^{(1)})^2]_{ab}[\mathcal{G}^{(1)}]_{ij}[\mathcal{G}^{(1)}]_{st}$ in the sequel. We have
\begin{align}\label{Ex1 expansion}
	&\E \left( \eta_0  \langle x_1,Y_ix_1 \rangle  \langle x_1,Z_ix_1 \rangle^2 \right) 
	=\eta_0\sum_{k_1,k_2,k_3}^{*}\E \prod_{i=1}^3x_{k_i1}^2 \Big( \mathcal{A}(k_1k_1,k_2k_2,k_3k_3)+2\mathcal{A}(k_1k_1,k_2k_3,k_2k_3)\notag \\ 
	&\qquad+4\mathcal{A}(k_1k_2,k_1k_2,k_3k_3)+8\mathcal{A}(k_1k_2,k_1k_3,k_2k_3) \Big) 
	+ \eta_0 \sum_{\mathcal{E}_6} \E\prod_{i=1}^6 x_{k_i1}\mathcal{A}(k_1k_2,k_3k_4,k_5k_6), 
\end{align}

where $\mathcal{E}_6$ denotes the set of indices $k_i \in \{1,\cdots,M \}, i = 1,\cdots,6$ such that $k_i$ appears even number of times and there is an index $k_i$ appears at least four times. Here we used the fact that the expectation equals to $0$ when there is an index $k_i$ appears odd number of times since $x_1$ is unconditionally distributed. Observe that 
$
	\# \mathcal{E}_6 \lesssim N^2.
$
Then using (\ref{GijBound}) and (\ref{G2ijBound}), we have  
\begin{align*}
	 \Big|\eta_0 \sum_{\mathcal{E}_6} \E\prod_{i=1}^6 x_{k_i1}\mathcal{A}(k_1k_2,k_3k_4,k_5k_6)\Big|  \lesssim N^{-4/3+\epsilon}.
\end{align*}
Replacing $x_1$ with $\tilde{x}_1$ in (\ref{Ex1 expansion}) and then taking the difference, we arrive at
\begin{align*}
	\E \left( \eta_0  \langle x_1,Y_ix_1 \rangle  \langle x_1,Z_ix_1 \rangle^2 \right) - \E \left( \eta_0  \langle \tilde{x}_1,Y_i\tilde{x}_1 \rangle  \langle \tilde{x}_1,Z_i\tilde{x}_1 \rangle^2 \right) =: \sum_{i=1}^4 2^{i-1}\Delta I_i + O_\prec( N^{-4/3+\epsilon}),
\end{align*}
where
\begin{align}
	\Delta I_1 := \eta_0\sum_{k_1,k_2,k_3}^{*} \bar{\Delta}_{k_1k_2k_3} \mathcal{A}(k_1k_1,k_2k_2,k_3k_3), \quad \Delta I_2 := \eta_0\sum_{k_1,k_2,k_3}^{*}\bar{\Delta}_{k_1k_2k_3} \mathcal{A}(k_1k_1,k_2k_3,k_2k_3), \notag\\
	\Delta I_3 := \eta_0\sum_{k_1,k_2,k_3}^{*} \bar{\Delta}_{k_1k_2k_3} \mathcal{A}(k_1k_2,k_1k_2,k_3k_3), \quad \Delta I_4 := \eta_0\sum_{k_1,k_2,k_3}^{*} \bar{\Delta}_{k_1k_2k_3} \mathcal{A}(k_1k_2,k_1k_3,k_2k_3), \label{121406}
\end{align}
with $\bar{\Delta}_{k_1k_2k_3} := \E ( \prod_{i=1}^3x_{k_i1}^2 -  \prod_{i=1}^3\tilde{x}_{k_i1}^2)$.  In fact, we are able to show that
\begin{align}\label{Bounds for Delta Ii}
	\Delta I_1,\; \Delta I_3 = O_\prec(N^{-7/6+\epsilon}),\quad \Delta I_2, \; \Delta I_4 = O_\prec(N^{-4/3+\epsilon}).
\end{align}
With these estimates, we obtain
\begin{align*}
	\E \left( \eta_0  \langle x_1,Y_ix_1 \rangle  \langle x_1,Z_ix_1 \rangle^2 \right) - \E \left( \eta_0  \langle \tilde{x}_1,Y_i\tilde{x}_1 \rangle  \langle \tilde{x}_1,Z_i\tilde{x}_1 \rangle^2 \right) = O_\prec(N^{-7/6+\epsilon}).
\end{align*}
We briefly state the idea of proving (\ref{Bounds for Delta Ii}). Since $\mathcal{G}_1^{(1)}$ is independent of $x_1$ and $\tilde{x}_1$, we only need to consider the bounds for the difference of mixed moments of $x_1$ and $\tilde{x}_1$. If $x_1$ and $\tilde{x}_1$ both have i.i.d coordinates, we can easily obtain that $\Delta I_i = 0$. However, $x_1$ and $\tilde{x}_1$ are isotropic log-concave random vectors in the current scenario. By Lemma \ref{approximation lemma}, we know that most of the mixed moments of $x_1$ and $\tilde{x}_1$ behave similarly to the i.i.d case. This observation enables us to obtain desired bounds for $\Delta I_i$'s. In the sequel, we show the detailed estimates for $\Delta I_i$'s.  

We first consider the estimation of $\Delta I_1$. From the definition in (\ref{121406}), we can derive 
\begin{align}\label{expansion of I1}
	\Delta I_1 
	 = &\eta_0(1 - (1-t_0)^3)\Delta I_{11} -  \eta_0\frac{(1-t_0)^2t_0}{N}\Delta I_{12}
	 \notag\\
	 &-\eta_0\Big(\frac{-3t_0^2(1-t_0)}{N^3}  + \frac{t_0^3}{N^3}\Big)\sum_{k_1,k_2,k_3}^{*}\mathcal{A}(k_1k_1,k_2k_2,k_3k_3),
\end{align}
where 
\begin{align*}
&\Delta I_{11}:=\E \sum_{k_1,k_2,k_3}^{*}x_{k_11}^2x_{k_21}^2x_{k_31}^2\mathcal{A}(k_1k_1,k_2k_2,k_3k_3), \notag\\ 
&\Delta I_{12}:= \E \sum_{k_1,k_2,k_3}^{*}(x_{k_11}^2x_{k_21}^2+ x_{k_21}^2x_{k_31}^2+x_{k_11}^2x_{k_31}^2)\mathcal{A}(k_1k_1,k_2k_2,k_3k_3).
\end{align*}
For $\Delta I_{11}$,  we can further split it into three terms.
\begin{align*}
	\Delta I_{11} 
	=&\E \sum_{k_1,k_2,k_3}^{*} \prod_{a=1}^3x_{k_a1}^2[(\mathcal{G}^{(1)}_1)^2]_{k_1k_1}m_1^2+\E \sum_{k_1,k_2,k_3}^{*}\prod_{a=1}^3x_{k_a1}^2[(\mathcal{G}^{(1)}_1)^2]_{k_1k_1}([\mathcal{G}^{(1)}_1]_{k_2k_2} - m_1)m_1\\
	&+\E \sum_{k_1,k_2,k_3}^{*}\prod_{a=1}^3x_{k_a1}^2[(\mathcal{G}^{(1)}_1)^2]_{k_1k_1}[\mathcal{G}^{(1)}_1]_{k_2k_2}([\mathcal{G}^{(1)}_1]_{k_3k_3} - m_1) \\
	=&: \Delta I_{111} +  \Delta I_{112} +  \Delta I_{113}.
\end{align*}

Next, we bound the above terms one by one. Note that $\sum_{k} \mathbb{E}x_{k1}^2=y$. We first rewrite $\Delta I_{111}$ as
\begin{align*}
	 \Delta I_{111} 
	 =&\frac{y^2}{N}\E \sum_{k_1}[(\mathcal{G}^{(1)}_1)^2]_{k_1k_1}m_1^2+y\E \sum_{k_1}x_{k_11}^2\Big(\sum_{k_2}x_{k_21}^2 - y\Big) [(\mathcal{G}^{(1)}_1)^2]_{k_1k_1}m_1^2\\
	 &+\E \sum_{k_1,k_2}x_{k_11}^2x_{k_21}^2\Big(\sum_{k_3}x_{k_31}^2 - y\Big) [(\mathcal{G}^{(1)}_1)^2]_{k_1k_1}m_1^2+ O(N^{-4/3+\epsilon}),
\end{align*}
where the error term comes from the replacement of $\sum^*$ by $\sum$. 
Using (\ref{G2ijBound}) together with the thin shell bound for unconditional log-concave random vector (cf. Lemma \ref{thin shell}), we have
\begin{align*}
	&\bigg|\E \sum_{k_1}x_{k_11}^2\Big(\sum_{k_2}x_{k_21}^2 - y\Big) [(\mathcal{G}^{(1)}_1)^2]_{k_1k_1} \bigg| \notag\\
	&\le \Bigg( \E \bigg|\sum_{k_1}x_{k_11}^2[(\mathcal{G}^{(1)}_1)^2]_{k_1k_1}\bigg|^2 \Bigg)^{1/2}\Bigg( \E \bigg|\sum_{k_2}x_{k_21}^2 - y\bigg|^2 \Bigg)^{1/2}\lesssim N^{-1/6+\epsilon},
\end{align*}
and similarly,
\begin{align*}
	\bigg| \E \sum_{k_1,k_2}x_{k_11}^2x_{k_21}^2\Big(\sum_{k_3}x_{k_31}^2 - y\Big) [(\mathcal{G}^{(1)}_1)^2]_{k_1k_1}\bigg| \lesssim N^{-1/6+\epsilon}.
\end{align*}
Therefore, we have
\begin{align*}
	\Delta I_{111} =\frac{y^2}{N} \E \sum_{k_1}[(\mathcal{G}^{(1)}_1)^2]_{k_1k_1}m_1^2 + O(N^{-1/6+\epsilon}).
\end{align*}

Next, we  rewrite $\Delta I_{112}$ as
\begin{align*}
	\Delta I_{112}
	=&\E \sum_{k_1,k_2} \sum_{k_3 \in B^+_{k_1,k_2}\cup B^-_{k_1,k_2}} \prod_{a=1}^3x_{k_a1}^2[(\mathcal{G}^{(1)}_1)^2]_{k_1k_1}([\mathcal{G}^{(1)}_1]_{k_2k_2} - m_1)m_1 \\
	&+\E \sum_{k_1,k_2} \sum_{k_3 \notin B^+_{k_1,k_2}\cup B^-_{k_1,k_2} } \prod_{a=1}^3x_{k_a1}^2[(\mathcal{G}^{(1)}_1)^2]_{k_1k_1}([\mathcal{G}^{(1)}_1]_{k_2k_2} - m_1)m_1+ O(N^{-2/3+\epsilon}). 
\end{align*}
Using (\ref{GijBound}) and (\ref{G2ijBound}) with Lemma \ref{approximation lemma}, we get
\begin{align*}
	\bigg|\E \sum_{k_1,k_2} \sum_{k_3 \in B^+_{k_1,k_2}\cup B^-_{k_1,k_2}}\prod_{a=1}^3x_{k_a1}^2[(\mathcal{G}^{(1)}_1)^2]_{k_1k_1}([\mathcal{G}^{(1)}_1]_{k_2k_2} - m_1)m_1\bigg| \lesssim N^{-1/3+\epsilon}.
\end{align*}
Further, we have
\begin{align*}
 &\E \sum_{k_1,k_2} \sum_{k_3 \notin B^+_{k_1,k_2}\cup B^-_{k_1,k_2} }x_{k_11}^2x_{k_21}^2x_{k_31}^2[(\mathcal{G}^{(1)}_1)^2]_{k_1k_1}([\mathcal{G}^{(1)}_1]_{k_2k_2} - m_1)m_1\notag\\
 &= y\mathbb{E}\sum_{k_1,k_2} x_{k_11}^2x_{k_21}^2[(\mathcal{G}^{(1)}_1)^2]_{k_1k_1}([\mathcal{G}^{(1)}_1]_{k_2k_2} - m_1)m_1+ O(N^{-1/3+\epsilon})
\end{align*}
Therefore,
\begin{align*}
	\Delta I_{112} = \frac{1}{N}\mathbb{E}\sum_{k_1,k_2,k_3} x_{k_11}^2x_{k_21}^2[(\mathcal{G}^{(1)}_1)^2]_{k_1k_1}([\mathcal{G}^{(1)}_1]_{k_2k_2} - m_1)m_1+ O(N^{-1/3+\epsilon}).
\end{align*}
Repeatedly using Lemma \ref{approximation lemma} to replace $x_{k_i 1}^2, i = 1,2$ with $N^{-1}$ up to negligible error, we finally get
\begin{align*}
	\Delta I_{112}
	=& \frac{1}{N^3} \E \sum_{k_1,k_2,k_3}[(\mathcal{G}^{(1)}_1)^2]_{k_1k_1}([\mathcal{G}^{(1)}_1]_{k_2k_2} - m_1)m_1 + O(N^{-1/3+\epsilon}),
\end{align*}
The estimation of $\Delta I_{113}$ is identical to that of $\Delta I_{112}$, and thus we omit the details. We have
\begin{align*}
	\Delta I_{113} = \frac{1}{N^3}\E \sum_{k_1,k_2,k_3}[(\mathcal{G}^{(1)}_1)^2]_{k_1k_1}[\mathcal{G}_1^{(1)}]_{k_2k_2}([\mathcal{G}^{(1)}_1]_{k_3k_3} - m_1)+ O(N^{-1/3+\epsilon}).
\end{align*}

Combining the estimates of $\Delta I_{111}, \Delta I_{112}$ and $\Delta I_{113}$ above, we  obtain
\begin{align}\label{estimate of I11}
	\Delta I_{11} = \frac{1}{N^3}\sum_{k_1,k_2,k_3}\mathcal{A}(k_1k_1,k_2k_2,k_3k_3)+ O(N^{-1/6+\epsilon}).
\end{align}

Similarly, we can also show $\Delta I_{12} = N^{-3}\sum_{k_1,k_2,k_3}\mathcal{A}(k_1k_1,k_2k_2,k_3k_3) + O(N^{-1/6+\epsilon})$. Plugging the estimates of $\Delta I_{11}$ and $\Delta I_{12}$ to (\ref{expansion of I1}), and using the fact that $\eta_0 t_0 \sim N^{-1+\epsilon}$, it is elementary to show
\begin{align*}
	\Delta I_1 = O(N^{-7/6+\epsilon}).
\end{align*}
The estimation of $\Delta I_3$ is similar, we omit the details. 

Next, we consider the estimation of $\Delta I_2$ and $\Delta I_4$. Similar to (\ref{expansion of I1}),
\begin{align}\label{expansion of I2}
	\Delta I_2 =& \eta_0 (1 - (1-t_0)^3)\E \sum_{k_1,k_2,k_3}^{*}x_{k_11}^2x_{k_21}^2x_{k_31}^2\mathcal{A}(k_1k_1,k_2k_3,k_2k_3) \notag \\
	 &-\eta_0\frac{(1-t_0)^2t_0}{N}\E \sum_{k_1,k_2,k_3}^{*}(x_{k_11}^2x_{k_21}^2+ x_{k_21}^2x_{k_31}^2+x_{k_11}^2x_{k_31}^2)\mathcal{A}(k_1k_1,k_2k_3,k_2k_3) \notag\\
	 &-\eta_0\left(\frac{-3t_0^2(1-t_0)}{N^3}  + \frac{t_0^3}{N^3}\right)\sum_{k_1,k_2,k_3}^{*}\mathcal{A}(k_1k_1,k_2k_3,k_2k_3) .
\end{align}
Notice that $|([\mathcal{G}^{(1)}]_{k_2k_3})^2| \prec N^{-2/3+\epsilon}$ when $k_2 \neq k_3$ (cf. (\ref{GijBound})), $\Delta I_2$ can be crudely bounded as $\Delta I_2 = O(N^{-4/3+\epsilon})$. Since $\Delta I_4$ also contains two off-diagonal entries $[\mathcal{G}^{(1)}_1]_{k_1k_3}$ and $[\mathcal{G}^{(1)}_1]_{k_2k_3}$, we have $\Delta I_4 = O(N^{-4/3+\epsilon})$. This concludes all the estimates.  

\end{proof}

\section{Spiked covariance matrices}\label{Sec Proofs for Spiked covariance matrices}
In this section, we give the proof of Theorem \ref{Spiked model distribution}. The following Lemma in \cite{BKYPCA} describes the location of the outlier eigenvalues. Although the original result is proved for $X$ with i.i.d entries, the whole proof only requires the local MP law and spectral rigidity, which are now available for our model.
\begin{lemma}
	Suppose the assumptions in Theorem \ref{Spiked model distribution} hold.  We have
	\begin{align*}
		|\lambda_i(Q) - \theta(d_i) | \prec N^{-1/2},
	\end{align*} 
	where 
	\begin{align*}
		\theta(z) := 1 + z + y + yz^{-1}, \qquad \text{for} \;  z \in \mathbb{C}, \; \Re z > \sqrt{y}.
	\end{align*}
\end{lemma}
Based on the location of the outlier eigenvalues, we can also have the following lemma for the representation of the eigenvalues. The proof of the following lemma can be found in \cite{BDWW}, again with the aid of the local MP law and spectral rigidity.  
\begin{lemma}\label{Phi description}
	Suppose the assumptions in Theorem \ref{Spiked model distribution} hold. We have
	\begin{align*}
		\lambda_i(Q) = \theta(d_i) - (d_i^2 - y)\theta(d_i)\left( [\mathcal{G}_1(\theta(d_i))]_{ii} - m_1(\theta(d_i)) \right) + O_{\prec}(N^{-\frac{1}{2}-\epsilon}), \quad i = 1,\cdots,r
	\end{align*}
	for some small fixed constant $\epsilon > 0$.
\end{lemma}

With the above two lemmas,  we now begin the proof of Theorem \ref{Spiked model distribution}. 
\begin{proof}[Proof of Theorem \ref{Spiked model distribution}]
	Define 
\begin{align}\label{def of Phi}
	\Phi_i := -\sqrt{N (d_i^2 - y)} \theta(d_i)\left( [\mathcal{G}_1(\theta(d_i))]_{ii} - m_1(\theta(d_i)) \right), \quad i = 1,\cdots,r.
\end{align} 
Let $r_i$ be the $i$-th row of $X$, and therefore $x_{ik}$ is the $k$-th entry of $r_i$. By the resolvent identity, we have 
\begin{align}\label{cal G expension}
	[\mathcal{G}_1(\theta(d_i))]_{ii} = \frac{1}{-\theta(d_i)(1 + \langle r_i, \mathcal{G}_2^{[i]}(\theta(d_i))r_i \rangle)} = m_1(\theta(d_i)) + \frac{\Theta_i}{\theta(d_i)(1+m_2(\theta(d_i)))^2} + O_{\prec}(\Theta_i^2).
\end{align}
where $\Theta_i := \langle r_i, \mathcal{G}_2^{[i]}(\theta(d_i))r_i \rangle - m_2(\theta(d_i))$. We claim that 
\begin{align}\label{Bound for Thetai}
	\Theta_i = O_\prec(N^{-1/2}).
\end{align}
The proof of (\ref{Bound for Thetai}) is given in Section \ref{Appendix Remaining estimation in Section 6}.
Having the bound for $\Theta_i$, (\ref{cal G expension}) can written as
\begin{align*}
	[\mathcal{G}_1(\theta(d_i))]_{ii} =m_1(\theta(d_i)) + \frac{\Theta_i}{\theta(d_i)(1+m_2(\theta(d_i)))^2} + O_{\prec}(N^{-1}).
\end{align*}
Plugging the above estimate into (\ref{def of Phi}), we have
\begin{align*}
	\Phi_i = \frac{ -\sqrt{d_i^2 - y}}{(1+m_2(\theta(d_i)))^2} \sqrt{N} \Theta_i + O_{\prec}(N^{-\frac12}), \quad i = 1,\cdots,r.
\end{align*}
Therefore, to obtain the asymptotic distribution of $\Phi_i$, it suffices to derive the asymptotic distribution of $\sqrt{N}\Theta_i$. In the sequel, for brevity, we omit $\theta(d_i)$ from the notations such as $[\mathcal{G}_2^{[i]}(\theta(d_i))]$ and $m_a(\theta(d_i))$. Under the assumption that $x_{ik}$'s are unconditionally distributed, we have
\begin{align*}
	\sqrt{N}\Theta_i &\overset{d}{=}\sqrt{N}\Big( \sum_{k,l} \delta_{k}\delta_{l}x_{ik}x_{il}[\mathcal{G}_2^{[i]}]_{kl} - \sum_{k}x_{ik}^2[\mathcal{G}_2^{[i]}]_{kk} \Big) + \sqrt{N}\Big( \sum_{k}x_{ik}^2[\mathcal{G}_2^{[i]}]_{kk}  - m_2 \Big) =: \mathcal{Q}_1 + \mathcal{Q}_2,
\end{align*}
where $\delta_i$'s are i.i.d. Rademacher variables, independent of $X$.  With certain abuse of notation, in this section, we also use the notation $\sum_{k_1, k_2,\ldots, k_\ell}$ to denote the sum over the index set $\{1, \ldots, N\}^\ell$ when there is no confusion. 
We first consider $\mathcal{Q}_2$. By the resolvent identity, we have
\begin{align*}
	[\mathcal{G}_2^{[i]}]_{kk} = -\frac{1}{\theta(d_i)(1 + \langle x_k^{[i]}, \mathcal{G}_1^{[i](k)}x_k^{[i]} \rangle)} = m_2 + \frac{\Gamma_{ik}}{\theta(d_i)(1 + ym_1)^2} + O_{\prec}(\Gamma_{ik}^2),
\end{align*}
where $x_k^{[i]}$ is the vector obtained by removing the $i$-th entry of $x_k$ and $\Gamma_{ik} = \langle x_k^{[i]},  \mathcal{G}_1^{[i](k)}x_k^{[i]} \rangle - ym_1$. Note that $x_k^{[i]}$ and $\mathcal{G}_1^{[i](k)}$ are independent, by the large deviation lemma, we can easily obtain $|\Gamma_{ik}| \prec N^{-1/2}$. Therefore, $\mathcal{Q}_2$ can be estimated as follows
\begin{align*}
	\mathcal{Q}_2 = \sqrt{N}\Big( \sum_{k}x_{ik}^2 - 1\Big)m_2 + \frac{ \sqrt{N}\left( \sum_{k}x_{ik}^2 \Gamma_{ik}\right)}{\theta(d_i)(1 + ym_1)^2} + O_{\prec}(N^{-\frac12}).
\end{align*}
Observe that the first term in $\mathcal{Q}_2$ is the sum of  i.i.d random variables. By the classical CLT, we know that this term converges to a Gaussian random variable. Now we focus on the second term of $\mathcal{Q}_2$. Further splitting it into diagonal part and off-diagonal part gives  
\begin{align}
	\sqrt{N} \Big( \sum_{k}x_{ik}^2 \Gamma_{ik} \Big) =& \sqrt{N} \sum_{k}x_{ik}^2\Big( \sum_{u, u\neq i}x_{uk}^2[ \mathcal{G}_1^{[i](k)}]_{uu} - y m_1 \Big) \notag\\
	&+ \sqrt{N} \sum_{k}x_{ik}^2\Big( \sum_{\substack{u,v\\ u,v \neq i, u\neq v}}x_{uk}x_{vk}[ \mathcal{G}_1^{[i](k)}]_{uv} \Big) =: \mathcal{Q}_{21} + \mathcal{Q}_{22}. \label{121420}
\end{align}
We estimate $\mathcal{Q}_{21}$ and $\mathcal{Q}_{22}$ by their mean and variance in Appendix \ref{sec est of Q21 Q22}, which gives
\begin{align*}
	\mathcal{Q}_{21} =  \sqrt{N}m_1 \sum_{k}x_{ik}^2 \sum_{u,u\neq i}\Big(x_{uk}^2 - \frac{1}{N}\Big) + o_p(1), \quad \mathcal{Q}_{22} = o_p(1).
\end{align*}
For the fluctuation of the first term in $\mathcal{Q}_{21}$, we have 
\begin{align}\label{Q2 estimates 2}
	\text{Var} \Big(\sqrt{N} \sum_{k}x_{ik}^2 \sum_{u,u\neq i}\Big(x_{uk}^2 - \frac{1}{N}\Big)\Big) = N\sum_k\text{Var} \Big( x_{ik}^2 \sum_{u,u\neq i}\Big(x_{uk}^2 - \frac{1}{N}\Big)\Big) \lesssim N^{-\frac16}.
\end{align}
The proof of (\ref{Q2 estimates 2}) is postponed to Section \ref{Appendix Remaining estimation in Section 6}.
Therefore, the second term in $\mathcal{Q}_2$ only contributes to the expectation part of the whole distribution of $\Theta_i$, and thus we arrive at 
\begin{align}\label{Q2 estimates 1}
	\mathcal{Q}_2 = \sqrt{N}m_2\Big( \sum_{k}x_{ik}^2 - 1\Big) + \frac{ \sqrt{N}m_1 }{\theta(d_i)(1 + ym_1)^2}  \mathbb{E}\Big[\sum_{k} x_{ik}^2 \sum_{u,u\neq i}\Big(x_{uk}^2 - \frac{1}{N}\Big)\Big]+ o_p(1).
\end{align}

Next, we consider the fluctuation of $\mathcal{Q}_1$.We have the following lemma by only considering the randomness of $\delta_{k}$'s. It is simply a CLT for the quadratic form of $\delta$-variables, conditioning on $x$-variables. 
\begin{lemma}\label{Estimates Q11}
Let $\E_{\delta}$ denote the expectation on $\delta_{k}$'s, we have for any $t \in \mathbb{R}$,
	\begin{align}\label{Estimates Q11 1}
	\E_{\delta} \bigg[\exp \bigg( \mathrm{i}t \frac{\mathcal{Q}_1}{\sqrt{b_1^2 + b_2^2}} \bigg)\bigg]  = \exp \bigg(\frac{-D_2t^2}{2(b_1^2+b_2^2)} \bigg) + O_{\prec}(N^{-c}),
\end{align}
for some constant $c > 0$. Here
\begin{align*}
			&b_1^2 = 2(m_2' - m_2^2),\;b_2^2 = m_2^2 \big(\E \big(\sqrt{N}x_{ik}\big)^4 - 1\big),\;D_2 = 2N\sum_{\substack{k,l,k\neq l}}x_{ik}^2x_{il}^2([\mathcal{G}_2^{[i]}]_{kl})^2.
\end{align*}
Furthermore,
\begin{align}\label{Estimates Q11 2}
	D_2 = 2m_2' - 2m_2^2 + o_p(1).
\end{align}
\end{lemma}
The proof of Lemma \ref{Estimates Q11} will be stated in Appendix \ref{sec proof of Lemma est Q11}. Set
\begin{align*}
	a =  \frac{ \sqrt{N}m_1(\theta(d_1)) }{\theta(d_i)(1 + dm_1)^2} \E\bigg[  \sum_{k}x_{ik}^2 \sum_{u\neq i}\Big(x_{uk}^2 - \frac{1}{N}\Big)\bigg].
\end{align*}
Our final goal is to prove 
\begin{align}\label{normal appro}
	\frac{\mathcal{Q}_1 + \mathcal{Q}_2 - a}{\sqrt{b_1^2 + b_2^2}} \Rightarrow \mathcal{N}(0, 1).
\end{align}
Let $\E_x$  denote the expectation w.r.t. $x_{ij}$'s. Then, considering the characteristic function of the above random variable, we have for any $t \in \mathbb{R}$,
\begin{align*}
	&\E \exp \bigg( \mathrm{i}t \frac{\mathcal{Q}_1 + \mathcal{Q}_2 - a}{\sqrt{b_1^2 + b_2^2}} \bigg) = \E_{x} \E_{\delta}  \exp \bigg( \mathrm{i}t \frac{\mathcal{Q}_1 + \mathcal{Q}_2 - a}{\sqrt{b_1^2 + b_2^2}} \bigg) \notag\\
	&= \E_{x}\bigg[\exp \bigg( \mathrm{i}t \frac{ \mathcal{Q}_2 - a}{\sqrt{b_1^2 + b_2^2}} \bigg) \E_{\delta} \bigg[\exp \bigg( \mathrm{i}t \frac{\mathcal{Q}_1}{\sqrt{b_1^2 + b_2^2}} \bigg)\bigg]  \bigg] \\
	&=\E_{x}\bigg[\exp \bigg( \mathrm{i}t \frac{ \mathcal{Q}_2 - a}{\sqrt{b_1^2 + b_2^2}} \bigg) \exp \bigg(\frac{-D_2t^2}{2(b_1^2+b_2^2)} \bigg) \bigg] + o(1).
\end{align*}
From (\ref{Q2 estimates 1}) and (\ref{Q2 estimates 2}), we know
\begin{align*}
	\mathcal{Q}_2 = \sqrt{N}m_2\left( \sum_{k}x_{ik}^2 - 1\right) + a + o_p(1).
\end{align*}
Since the above $x_{ik}$'s are i.i.d., we have
\begin{align*}
	\frac{\mathcal{Q}_2 - a}{\sqrt{b_1^2+b_2^2}}  \Rightarrow  \mathcal{N}\bigg(0, \frac{b_2^2}{b_1^2 + b_2^2}\bigg).
\end{align*}
We then recall the following basic lemma. 
\begin{lemma}\label{lemma 5}
	For $n \ge 1$, let $U_n$, $T_n$ be r.v.s satisfying the following conditions:
(i) $U_n \overset{p}{\rightarrow} a$;  
(ii) $\{ T_n \}$ and $\{T_n U_n \}$ are uniformly integrable sequence,
(iii) $\E T_n \rightarrow c $.
	Then we have $\E T_nU_n \rightarrow ac$ as $n \to \infty$.
\end{lemma}
Let 
$
	U_N := \exp (-D_2t^2/ (2b_1^2+2b_2^2) )$ and $T_N := \exp ( \mathrm{i}t  (\mathcal{Q}_2- a)/\sqrt{b_1^2 + b_2^2} ). 
$
It is obvious that $\{ T_N \}$ and $\{T_N U_N \}$ are uniformly integrable sequence since they are both uniformly bounded by $1$.  The other conditions are readily proved. Hence, by Lemma \ref{lemma 5}, we can get
\begin{align*}
	\E \exp \bigg( \mathrm{i}t \frac{\mathcal{Q}_1 + \mathcal{Q}_2 - a}{\sqrt{b_1^2 + b_2^2}} \bigg) 
	\to  \exp \bigg(-\frac{t^2}{2} \bigg),\quad \text{as} \quad N \to \infty.
\end{align*}
This completes the proof of (\ref{normal appro}). 
Therefore, we complete the proof of Theorem \ref{Spiked model distribution} after some elementary simplification using (\ref{identitym1m2}). 
\end{proof}

\section{Technical estimates in Section \ref{Sec Proofs for Spiked covariance matrices}}\label{Appendix Remaining estimation in Section 6}
\subsection{Bound for $\Theta_i$} Recall the definition of $\Theta_i$ from (\ref{cal G expension}). We can rewrite it as
\begin{align}
	\Theta_i =E_1+E_2+\sum_{k,l, k \neq l}x_{ik}x_{il}[ \mathcal{G}_2^{[i]}(\theta(d_i))]_{kl}, \label{121410}
\end{align}
where
\begin{align*}
E_1:=\sum_{k}x_{ik}^2\left([ \mathcal{G}_2^{[i]}(\theta(d_i))]_{kk}- m_2(\theta(d_i)) \right), \qquad E_2:=\sum_{k}\left( x_{ik}^2 -\frac{1}{N} \right)m_2(\theta(d_i)). 
\end{align*}
Using entry-wise local law (cf. (\ref{entry wise local law})), we have $|E_1| \prec N^{-1/2}$. Since $x_{ik}$'s are i.i.d random variables, we can easily obtain $|E_2| \prec  N^{-1/2}$. Hence, what remains is the estimate of the last term in (\ref{121410}). For brevity, we omit the variable $\theta(d_i)$ from the notations such as $\mathcal{G}_2^{[i]}(\theta(d_i))$ and  $m_2(\theta(d_i))$ in the sequel. Since $x_{ik}$'s are unconditionally distributed, we have 
\begin{align*}
	\sum_{k,l,k \neq l}x_{ik}x_{il}[ \mathcal{G}_2^{[i]}]_{kl} \overset{d}{=} \sum_{k,l,k \neq l}\delta_k\delta_l x_{ik}x_{il}[ \mathcal{G}_2^{[i]}]_{kl} 
\end{align*}
where $\delta_k$'s are i.i.d Bernoulli random variables. Applying large deviation inequality for the quadratic form of $\delta_i$'s, we have 
 \begin{align*}
	\Big| \sum_{k,l, k \neq l}\delta_k\delta_l x_{ik}x_{il}[ \mathcal{G}_2^{[i]}]_{kl} \Big| \prec \Big| \sum_{k,l, k \neq l} x^2_{ik}x^2_{il}|[ \mathcal{G}_2^{[i]}]_{kl}|^2 \Big|^{\frac{1}{2}} \prec N^{-\frac12}.
\end{align*}
Therefore, we can conclude that $|\Theta_i| \prec  N^{-1/2}$.
\subsection{Estimates of $\mathcal{Q}_{21}$ and $\mathcal{Q}_{22}$}\label{sec est of Q21 Q22}
We first consider the estimation of $\mathcal{Q}_{21}$. We first rewrite 
\begin{align*}
	\mathcal{Q}_{21} =& \sqrt{N}m_1 \sum_{k}x_{ik}^2 \sum_{u,u\neq i}\Big(x_{uk}^2 - \frac{1}{N}\Big) + \sqrt{N} \sum_{k}x_{ik}^2 \sum_{u,u\neq i}\Big(x_{uk}^2 - \frac{1}{N}\Big)(  \mathcal{G}_1^{[i](k)}]_{uu} - m_1) \notag\\
	&+ \sqrt{N}y \sum_{k}x_{ik}^2\Big( \frac{1}{M}\Tr \mathcal{G}_1^{[i](k)} - m_1 \Big).
\end{align*}
Using the  local law (cf. Theorem \ref{Strong MP law}), we see that the last term above is $O_{\prec}(N^{-\frac12})$. 
For the second term, we have 
\begin{align*}
	&\E \Big| \sum_{u,u\neq i}\Big(x_{uk}^2 - \frac{1}{N}\Big)(  \mathcal{G}_1^{[i](k)}]_{uu} - m_1) \Big|^2 \\
	=& \sum_{\substack{u_1,u_2\\ u_1 , u_2 \neq i}} \E \Big( x_{u_1k}^2 -\frac{1}{N} \Big)\Big( x_{u_2k}^2 -\frac{1}{N} \Big) \E (  \mathcal{G}_1^{[i](k)}]_{u_1u_1} - m_1)(  \mathcal{G}_1^{[i](k)}]_{u_2u_2} - m_1)\\
	\leq & \sum_{\substack{ u_1, u_2\\ u_1 \neq i, u_2 \neq i}} O_\prec(\frac{1}{N})\E \Big( x_{u_1k}^2x_{u_2k}^2 -\frac{1}{N^2} \Big) \prec N^{-\frac43},
	\end{align*}
	where in the last step, we again used Lemma \ref{approximation lemma}. 
Hence, we have
\begin{align*}
	\mathcal{Q}_{21} =  \sqrt{N}m_1 \sum_{k}x_{ik}^2 \sum_{u\neq i}\Big(x_{uk}^2 - \frac{1}{N}\Big) + o_p(1), 
	\end{align*}

Next, we consider the estimation of $\mathcal{Q}_{22}$.
Under the unconditional assumption, we have
\begin{align*}
	\mathcal{Q}_{22} \overset{d}{=}\sqrt{N} \sum_{k}x_{ik}^2\Big( \sum_{\substack{u,v\\ u,v \neq i, u\neq v}}\delta_{uk}\delta_{vk}x_{uk}x_{vk} \mathcal{G}_1^{[i](k)}]_{uv} \Big),
\end{align*}
where $\delta_{uk}$'s are i.i.d Rademacher random variables. Then taking the expectation w.r.t. the $\delta$-variables simply leads to
\begin{align*}
	\E \mathcal{Q}_{22}  = 0,
\end{align*}
and
\begin{align*}
	\text{Var} \mathcal{Q}_{22}  
	=&N \sum_{k}\E  x_{ik}^4 \bigg( \sum_{\substack{u,v\\ u,v \neq i, u\neq v}}x_{uk}^2x_{vk}^2 \mathcal{G}_1^{[i](k)}]_{uv}^2 \bigg) = O_\prec(N^{-1}).
\end{align*}
Therefore, we have $\mathcal{Q}_{22} =o_p(1)$.

\subsection{Proof of (\ref{Q2 estimates 2})}
Write
\begin{align*}
	\text{Var} \Big( x_{ik}^2 \sum_{u,u\neq i}\Big(x_{uk}^2 - \frac{1}{N}\Big)\Big) = \E \Big( x_{ik}^2 \sum_{u,u\neq i}\Big(x_{uk}^2 - \frac{1}{N}\Big)\Big)^2 - \Big(\E \Big( x_{ik}^2 \sum_{u,u\neq i}\Big(x_{uk}^2 - \frac{1}{N}\Big)\Big)\Big)^2.
\end{align*}
For the second term, we can directly bound it by Lemma \ref{approximation lemma}, which gives
\begin{align*}
	\E \Big( x_{ik}^2 \sum_{u,u\neq i}\Big(x_{uk}^2 - \frac{1}{N}\Big)\Big) \lesssim \frac{1}{N^{\frac{4}{3}}}.
\end{align*}
For the first term, we have by Cauchy-Schwarz inequality,
\begin{align*}
	\E \Big( x_{ik}^2 \sum_{u,u\neq i}\Big(x_{uk}^2 - \frac{1}{N}\Big)\Big)^2 \le &  \Big( \E \Big(x_{ik}^4\sum_{u,u\neq i}\Big(x_{uk}^2 - \frac{1}{N}\Big)   \Big)^2\Big)^{1/2} \Big( \E \Big( \sum_{u,u\neq i}\Big(x_{uk}^2 - \frac{1}{N}\Big) \Big)^2 \Big)^{1/2} \\
	\lesssim & \frac{1}{N^2}\Big( \sum_{u_1,u_1\neq i}\sum_{u_2,u_2\neq i}\E  \Big(x_{u_1k}^2x_{u_2k}^2 - \frac{1}{N^2}\Big)  \Big)^{1/2} \lesssim \frac{1}{N^{\frac{13}{6}}},
\end{align*}
where in the last step we used Lemma \ref{approximation lemma} again to estimate the mixed moments. We conclude the proof of (\ref{Q2 estimates 2}) by combining the above estimates.

\subsection{Proof of Lemma \ref{Estimates Q11}}\label{sec proof of Lemma est Q11}
In this section, we give the proof of Lemma \ref{Estimates Q11}. We first consider (\ref{Estimates Q11 1}). Let 
\begin{align*}
	a_k := \frac{\delta_k}{\sqrt{N}},\quad  A_{kl} := Nx_{ik}x_{il}[\mathcal{G}_2^{[i]}]_{kl}.
\end{align*}
Then $A = (A_{kl})_{k,l=1}^{N}$ is an $N \times N$ symmetric matrix, and $A_{kl} = O_{\prec}(N^{-1/2})$ if $k \neq l$ and $ A_{kk} = O_{\prec}(1)$. 
For brevity, we will use the notation $\partial_k := \partial/\partial a_k$.
Rewrite $\mathcal{Q}_1$ as the following,
\begin{align*}
	\mathcal{Q}_1 = \sqrt{N}\Big( \sum_{k,l} a_{k}a_{l}A_{kl} - \frac{1}{N}\sum_{k}A_{kk} \Big).
\end{align*}
We set the characteristic function $\Phi(t) = \E_{\delta} \phi(t)$  with $\phi(t)=\exp ( \mathrm{i}t \mathcal{Q}_1/\sqrt{b_1^2 + b_2^2} )$. Then we can write its derivative as 
\begin{align*}
	\Phi'(t) 
	 =& \frac{\mathrm{i}\sqrt{N} }{{\sqrt{b_1^2 + b_2^2}} } \sum_{k,l}\E_{\delta} \big[ a_{k}a_{l}A_{kl} \phi(t)\big]  - \frac{\mathrm{i}}{{\sqrt{b_1^2 + b_2^2}\sqrt{N} } } \sum_{k}\E_{\delta} \big[ A_{kk}  \phi(t)\big].
\end{align*}
Observe that if  $k = l$, we have $a_k^2=1$. 
Therefore,
\begin{align*}
	\Phi'(t) =\frac{\mathrm{i}\sqrt{N} }{{\sqrt{b_1^2 + b_2^2}} } \sum_{k,l,k\neq l}\E_{\delta} \big[ a_{k}a_{l}A_{kl}  \phi(t)\big].
\end{align*}

For $k \neq l$, performing cumulant expansion w.r.t. $a_{k}$ gives,
\begin{align}
	\sum_{k,l,k \neq l}\E_{\delta} \big[ a_{k}a_{l}A_{kl}  \phi(t)\big] = \frac{1}{N} \sum_{k,l,k \neq l}\E_{\delta} \big[ a_{l}A_{kl} \partial_k  \phi(t) \big] - \frac{1}{3N^2}\sum_{k,l,k \neq l}\E_{\delta} \big[ a_{l}A_{kl} \partial_k^3 \phi(t)\big] + \mathcal{R}.  \label{121460}
\end{align}
The remainder term $\mathcal{R}$ can be easily bounded as $\mathcal{R} = O_\prec(N^{-1/2})$, we omit the details.

For the first  term in the RHS of (\ref{121460}), by direct calculation, we have
\begin{align*}
	&\frac{1}{N} \sum_{k,l,k \neq l}\E_{\delta} \big[ a_{l}A_{kl} \partial_k  \phi(t) \big] 
	=\frac{2\mathrm{i}t}{ \sqrt{N}\sqrt{b_1^2 + b_2^2}} \bigg(\frac{1}{N}  \sum_{k,l,k \neq l} A_{kl}^2\bigg)  \E_{\delta} [    \phi(t)] \\
	&\quad \quad \quad + \frac{2\mathrm{i}t}{ \sqrt{N}\sqrt{b_1^2 + b_2^2}} \sum_{j,l} \E _{\delta} \Big[ \sum_{k,k \neq l}\big(a_{l}a_j  A_{kl} A_{jk} -\E _{\delta} \big[ a_{l}a_j A_{kl} A_{jk}\big] \big)    \phi(t)\Big].
\end{align*}
Notice that by the large deviation inequality for quadratic form of $a_i$'s, 
\begin{align*}
 &\sum_{j,l} \sum_{k,k \neq l}\Big(a_{l}a_j  A_{kl} A_{jk} -\E _{\delta} \big[ a_{l}a_j  A_{kl} A_{jk}\big] \Big)\prec  \frac{1}{N} \bigg(\sum_{l,j}  \Big|\sum_{k, k \neq l}A_{kl} A_{jk}\Big|^2 \bigg)^{1/2} \\
 &\qquad\qquad\prec  \frac{1}{N} (\Tr A^2)^{1/2} + \frac{1}{\sqrt{N}}  \prec \frac{1}{\sqrt{N}}.
\end{align*}
Therefore, we have
\begin{align*}
	\frac{1}{N} \sum_{k,l,k \neq l}\E_{\delta} \big[ a_{l}A_{kl} \partial_k \phi(t)\big] = \frac{2\mathrm{i}t}{ \sqrt{N}\sqrt{b_1^2 + b_2^2}} \bigg(\frac{1}{N}  \sum_{k \neq l} A_{kl}^2\bigg)  \E_{\delta} [  \phi(t)] +  O_{\prec} (N^{-\frac12}).
\end{align*}

For the second  term in (\ref{121460}), we have
\begin{align*}
	&\frac{1}{3N^2}\sum_{k,l,k \neq l}\E_{\delta} \big[ a_{l}A_{kl} \partial_k^3  \phi(t)\big] 
	=\frac{-4 t^2}{N (b_1^2 + b_2^2)}\sum_{k,l,k \neq l}\E_{\delta} \Big[ a_{l}A_{kl}A_{kk} \big(\sum_{j} a_j  A_{jk} \big)  \phi(t) \Big] \\
	&-\frac{8 \mathrm{i}t^3 }{3\sqrt{N}(b_1^2 + b_2^2)^{3/2}}\sum_{k,l,k \neq l}\E_{\delta} \Big[ a_{l}A_{kl} \big(\sum_{j} a_j  A_{jk} \big)^3   \phi(t) \Big]= O_{\prec}(N^{-1} ).
\end{align*}
where in the last step we used
$
	\sum_{j} a_j  A_{jk} = O_{\prec}( N^{-1/2}).
$
Combining the above estimates, we have
\begin{align*}
	\Phi'(t) = -\frac{2t^2}{ b_1^2 + b_2^2} \bigg(\frac{1}{N}  \sum_{k,l,k \neq l} A_{kl}^2\bigg) \Phi(t) +  O_{\prec}(N^{-\frac12}).
\end{align*}
Then (\ref{Estimates Q11 1}) follows by solving the differential equation.

Next, we give the proof of (\ref{Estimates Q11 2}).
Notice that by Theorem \ref{Strong MP law}, we have
\begin{align*}
	&N\sum_{k,l, k\neq l}x_{ik}^2x_{il}^2([\mathcal{G}_2^{[i]}]_{kl})^2 
	 = m_2' - m_2^2+N\sum_{k,l,k\neq l}\left(x_{ik}^2x_{il}^2  - \frac{1}{N^2}\right)([\mathcal{G}_2^{[i]}]_{kl})^2 + O_\prec(N^{-\frac12}).
\end{align*}
Hence, what remains is to show 
\begin{align}\label{D2 estimates}
	N\sum_{k,l,k\neq l}\left(x_{ik}^2x_{il}^2  - \frac{1}{N^2}\right)([\mathcal{G}_2^{[i]}]_{kl})^2 = o_p(1).
\end{align}

To show (\ref{D2 estimates}), it suffices to estimate the expectation and variance of the LHS of (\ref{D2 estimates}). 
We first consider the  expectation. For simplicity, in the sequel, we adopt the following shorthand notation
\begin{align*}
\sum_{u,v}^{(i)}:= \sum_{\substack{u,v \\ u,v\neq i}}.  
\end{align*}
Using the following resolvent expansion
\begin{align}\label{re expan}
	[\mathcal{G}_2(z)]_{kl}= z[\mathcal{G}_2(z)]_{kk}[\mathcal{G}_2^{(k)}(z)]_{ll}\Big(\sum_{u,v}x_{uk}x_{vl}[\mathcal{G}_1^{(kl)}(z)]_{uv} \Big),
\end{align}
we have
\begin{align}
	&\E \bigg[ \Big( x_{ik}^2x_{il}^2  - \frac{1}{N^2}\Big)([\mathcal{G}_2^{[i]}]_{kl})^2\bigg] \notag \\
	=& \theta^2(d_i)\E \bigg[ \Big( x_{ik}^2x_{il}^2  - \frac{1}{N^2}\Big)([\mathcal{G}_2^{[i]}]_{kk})^2([\mathcal{G}_2^{[i](k)}]_{ll})^2\Big( \sum_{u,v}^{(i)}x_{uk}x_{vl}\mathcal{G}_1^{[i](kl)}]_{uv} \Big)^2 \bigg] \notag \\
	=&\theta^2(d_i) m_2^4\E \bigg[\Big( x_{ik}^2x_{il}^2  - \frac{1}{N^2}\Big) \Big( \sum_{u,v}^{(i)}x_{uk}x_{vl}[\mathcal{G}_1^{[i](kl)}]_{uv} \Big)^2 \bigg] + O_{\prec}(N^{-\frac{7}{2}})   \notag \\
	=&\theta^2(d_i) m_2^4\E \bigg[\Big( x_{ik}^2x_{il}^2  - \frac{1}{N^2}\Big)  \sum_{u,v}^{(i)}x_{uk}^2x_{vl}^2([\mathcal{G}_1^{[i](kl)}]_{uv})^2\bigg] +  O_{\prec}(N^{-\frac{7}{2}}) , \notag
\end{align}
where in the second equality we used Theorem \ref{Strong MP law} to replace the diagonal entries of $[\mathcal{G}_2^{[i]}]_{kk}$ and $[\mathcal{G}_2^{[i](k)}]_{ll}$ with $m_2$, and in the last step we used the condition that $x_{ij}$'s are unconditionally distributed.
Notice that 
\begin{align*}
	&\E \bigg[\Big( x_{ik}^2x_{il}^2  - \frac{1}{N^2}\Big)  \sum_{u,v}^{(i)}x_{uk}^2x_{vl}^2([\mathcal{G}_1^{[i](kl)}]_{uv})^2\bigg] 
	=\sum_{u,v}^{(i)} \Big(\E \big[ x_{ik}^2x_{uk}^2\big]\E\big[x_{il}^2 x_{vl}^2 \big] - \frac{1}{N^4}\Big) \E \big[ ([\mathcal{G}_1^{[i](kl)}]_{uv})^2 \big].
\end{align*}
By Lemma \ref{approximation lemma} and the estimates of $([\mathcal{G}_1^{[i](kl)}]_{uv})^2$ in Theorem \ref{Strong MP law}, we have
\begin{align*}
	\E \bigg[\Big( x_{ik}^2x_{il}^2  - \frac{1}{N^2}\Big)  \sum_{u,v}^{(i)}x_{uk}^2x_{vl}^2([\mathcal{G}_1^{[i](kl)}]_{uv})^2\bigg]  = O_{\prec}( N^{-\frac{10}{3}} ).
\end{align*}
Therefore,
\begin{align*}
N\sum_{k,l,k\neq l}	\E \bigg[\left(x_{ik}^2x_{il}^2  - \frac{1}{N^2}\right)([\mathcal{G}_2^{[i]}]_{kl})^2\bigg] = O_{\prec}(N^{-\frac{1}{3}}).
\end{align*}

Next, we consider the second moment of the LHS of (\ref{D2 estimates}). Applying the resolvent expansion (\ref{re expan}) to $[\mathcal{G}_2^{[i]}(z)]_{kl}$ and further using 
Theorem \ref{Strong MP law}, we have
\begin{align*}
	&\E  \bigg( \sum_{k,l,k\neq l}\Big( x_{ik}^2x_{il}^2  - \frac{1}{N^2}\Big)([\mathcal{G}_2^{[i]}]_{kl})^2 \bigg)^2\\
	&=\theta^4(d_i) m_2^8 \E  \bigg( \sum_{k,l,k\neq l}\Big( x_{ik}^2x_{il}^2  - \frac{1}{N^2}\Big) \Big(\sum_{u,v}^{(i)}x_{uk}x_{vl}[\mathcal{G}_1^{[i](kl)}]_{uv} \Big)^2  \bigg)^2  + O_\prec(N^{-3}).
\end{align*}
Also we have, 
\begin{align*}
	&\E \bigg[ \bigg( \sum_{k,l,k\neq l}\Big( x_{ik}^2x_{il}^2  - \frac{1}{N^2}\Big) \Big(\sum_{u,v}^{(i)}x_{uk}x_{vl}[\mathcal{G}_1^{[i](kl)}]_{uv} \Big)^2  \bigg)^2  \bigg] \\
	&=\E \bigg[ \bigg( \sum_{k_1,l_1,k_1\neq l_1}\Big( x_{ik_1}^2x_{il_1}^2  - \frac{1}{N^2}\Big) \Big(\sum_{u_1,v_1}^{(i)}\sum_{u_2,v_2}^{(i)}x_{u_1k_1}x_{v_1l_1}x_{u_2k_1}x_{v_2l_1}[\mathcal{G}_1^{[i](k_1l_1)}]_{u_1v_1}[\mathcal{G}_1^{[i](k_1l_1)}]_{u_2v_2} \Big)  \bigg)\\
	&\qquad   \times\bigg( \sum_{k_2,l_2,k_2\neq l_2}\Big( x_{ik_2}^2x_{il_2}^2  - \frac{1}{N^2}\Big) \Big(\sum_{u_3,v_3}^{(i)}\sum_{u_4,v_4}^{(i)}x_{u_3k_2}x_{v_3l_2}x_{u_4k_2}x_{v_4l_2}[\mathcal{G}_1^{[i](k_2l_2)}]_{u_3v_3} [\mathcal{G}_1^{[i](k_2l_2)}]_{u_4v_4} \Big) \bigg)  \bigg].
\end{align*}
We only have to consider the case of $\{k_1,l_1 \} \cap \{k_2,l_2\}  = \emptyset$, since the other cases have fewer summands which will collectively give smaller order errors. Crudely bounding the other cases gives
\begin{align*}
	&\E \bigg[ \bigg( \sum_{k,l,k\neq l}\Big( x_{ik}^2x_{il}^2  - \frac{1}{N^2}\Big) \Big(\sum_{u,v}^{(i)}x_{uk}x_{vl}[\mathcal{G}_1^{[i](kl)}]_{uv} \Big)^2  \bigg)^2  \bigg]\\
	=& \sum_{k_1\neq l_1}\sum_{\substack{k_2\neq l_2\\\{k_1,l_1 \} \cap \{k_2,l_2\}  = \emptyset}}\sum_{u_1,v_1}^{(i)} \sum_{u_3,v_3}^{(i)} \E \bigg[  \Big( x_{ik_1}^2x_{il_1}^2  - \frac{1}{N^2}\Big)x_{u_1k_1}^2x_{v_1l_1}^2  [\mathcal{G}_1^{[i](k_2l_2)}]_{u_3v_3}^2 \\
	&\qquad     \times \Big( x_{ik_2}^2x_{il_2}^2  - \frac{1}{N^2}\Big)   x_{u_3k_2}^2x_{v_3l_2}^2[\mathcal{G}_1^{[i](k_1l_1)}]_{u_1v_1}^2    \bigg] + O_{\prec}(N^{-3} ),
\end{align*}
where we also used the unconditional assumption. 
Further, by the Woodbury matrix identity, we have
\begin{align*}
	[\mathcal{G}_1^{[i](k_2l_2)}]_{u_3v_3} =  [\mathcal{G}_1^{[i](k_2l_2k_1)}]_{u_3v_3} -  \frac{\mathrm{e}_{u_3}^*\mathcal{G}_1^{[i](k_2l_2k_1)}x_{k_1}x_{k_1}^*\mathcal{G}_1^{[i](k_2l_2k_1)}\mathrm{e}_{v_3} }{1 + x_{k_1}^*\mathcal{G}_1^{[i](k_2l_2k_1)}x_{k_1}}.
\end{align*}
Notice that by large deviation lemma for log-concave random vector (cf. Lemma \ref{Large deviation lemma}), we have
\begin{align*}
	|\mathrm{e}_{u_3}^*\mathcal{G}_1^{[i](k_2l_2k_1)}x_{k_1}| = \Big|\sum_{j}[\mathcal{G}_1^{[i](k_2l_2k_1)}]_{u_3j}x_{jk_1}\Big| \prec \frac{1}{\sqrt{N}} \bigg(\sum_{j} |[\mathcal{G}_1^{[i](k_2l_2k_1)}]_{u_3j}|^2 \bigg)^{1/2} \prec N^{-\frac12}. 
\end{align*}
Therefore, 
\begin{align*}
	[\mathcal{G}_1^{[i](k_2l_2)}]_{u_3v_3}^2 = [\mathcal{G}_1^{[i](k_2l_2k_1)}]_{u_3v_3}^2 + O_{\prec}\bigg( \frac{|[\mathcal{G}_1^{[i](k_2l_2k_1)}]_{u_3v_3}|}{N} \bigg).
\end{align*}
Repeating the above procedure, we get
\begin{align*}
	[\mathcal{G}_1^{[i](k_2l_2)}]_{u_3v_3}^2 = [\mathcal{G}_1^{[i](k_2l_2k_1l_1)}]_{u_3v_3}^2 + O_{\prec}\bigg( \frac{|[\mathcal{G}_1^{[i](k_2l_2k_1)}]_{u_3v_3}|+|[\mathcal{G}_1^{[i](k_2l_2k_1l_1)}]_{u_3v_3}|}{N} \bigg).
\end{align*}
Similarly,
\begin{align*}
	[\mathcal{G}_1^{[i](k_1l_1)}]_{u_1v_1}^2 = [\mathcal{G}_1^{[i](k_2l_2k_1l_1)}]_{u_1v_1}^2 + O_{\prec}\bigg( \frac{|[\mathcal{G}_1^{[i](k_1l_1k_2)}]_{u_1v_1}|+|[\mathcal{G}_1^{[i](k_2l_2k_1l_1)}]_{u_1v_1}|}{N} \bigg).
\end{align*}
With these estimates, we have
\begin{align*}
	&\E \bigg[ \bigg( \sum_{k,l,k\neq l}\Big( x_{ik}^2x_{il}^2  - \frac{1}{N^2}\Big) \Big(\sum_{u,v}^{(i)}x_{uk}x_{vl}[\mathcal{G}_1^{[i](kl)}]_{uv} \Big)^2  \bigg)^2  \bigg] \\
	=& \sum_{k_1,l_1,k_1\neq l_1}\sum_{\substack{k_2\neq l_2\\\{k_1,l_1 \} \cap \{k_2,l_2\}  = \emptyset}}\sum_{u_1,v_1}^{(i)} \sum_{u_3,v_3}^{(i)} \E \bigg[  \Big( x_{ik_1}^2x_{il_1}^2  - \frac{1}{N^2}\Big)x_{u_1k_1}^2x_{v_1l_1}^2  [\mathcal{G}_1^{[i](k_2l_2k_1l_1)}]_{u_3v_3}^2\bigg] \\
	&\qquad      \times\E\bigg[  \Big( x_{ik_2}^2x_{il_2}^2  - \frac{1}{N^2}\Big)   x_{u_3k_2}^2x_{v_3l_2}^2[\mathcal{G}_1^{[i](k_1l_1k_2l_2)}]_{u_1v_1}^2    \bigg] + O_{\prec}(N^{-\frac52} ) \\
	=&\sum_{k_1,l_1,k_1\neq l_1}\sum_{\substack{k_2\neq l_2\\\{k_1,l_1 \} \cap \{k_2,l_2\}  = \emptyset}}\sum_{u_1,v_1}^{(i)}  \Big(\E \big[  x_{ik_1}^2x_{u_1k_1}^2\big]\E\big[ x_{il_1}^2  x_{v_1l_1}^2\big]- \frac{1}{N^4}\Big)\E\big[ [\mathcal{G}_1^{[i](k_1l_1k_2l_2)}]_{u_1v_1}^2    \big]\\
	&\qquad     \times \sum_{u_3,v_3}^{(i)}\Big( \E\big[   x_{ik_2}^2x_{u_3k_2}^2 \big] \E \big[x_{il_2}^2x_{v_3l_2}^2\big]  - \frac{1}{N^4}   \Big)  \E \big[[\mathcal{G}_1^{[i](k_2l_2k_1l_1)}]_{u_3v_3}^2\big]  + O_{\prec}(N^{-\frac52} ) \\
	\prec& \frac{1}{N^2} \sum_{k_1,l_1,k_1\neq l_1}\sum_{u_1,v_1}^{(i)}  \Big| \E \big[  x_{ik_1}^2x_{u_1k_1}^2\big]\E\big[ x_{il_1}^2  x_{v_1l_1}^2\big]- \frac{1}{N^4}\Big| \\
	&\qquad   \qquad \times\sum_{k_2,l_2,k_2 \neq l_2}\sum_{u_3,v_3}^{(i)}\Big| \E\big[   x_{ik_2}^2x_{u_3k_2}^2 \big] \E \big[x_{il_2}^2x_{v_3l_2}^2\big]  - \frac{1}{N^4}   \Big|  + O_{\prec}(N^{-\frac52} ) \\
	=& \frac{1}{N^2}  \bigg(\sum_{k_1,l_1,k_1\neq l_1}\sum_{u_1,v_1}^{(i)}  \Big| \E \big[  x_{ik_1}^2x_{u_1k_1}^2\big]\E\big[ x_{il_1}^2  x_{v_1l_1}^2\big]- \frac{1}{N^4}\Big|\bigg)^2+ O_{\prec}(N^{-\frac52} ).
\end{align*}
By Lemma \ref{approximation lemma}, for any $k_1 \neq l_1$, we have
\begin{align*}
\sum_{u_1,v_1}^{(i)}  \Big| \E \big[  x_{ik_1}^2x_{u_1k_1}^2\big]\E\big[ x_{il_1}^2  x_{v_1l_1}^2\big]- \frac{1}{N^4}\Big| \lesssim N^{-\frac{7}{3}}
\end{align*}
Therefore, we obtain,
\begin{align*}
	\E \bigg[ \bigg( \sum_{k,l,k\neq l}\Big( x_{ik}^2x_{il}^2  - \frac{1}{N^2}\Big)([\mathcal{G}_2^{[i]}]_{kl})^2 \bigg)^2\bigg] \prec N^{-\frac83}+N^{-\frac52}\prec N^{-\frac52}.  
\end{align*}
Together with the estimate of expectation, we get
\begin{align*}
	N\sum_{k,l,k\neq l}\left(x_{ik}^2x_{il}^2  - \frac{1}{N^2}\right)([\mathcal{G}_2^{[i]}]_{kl})^2  = o_p(1).
\end{align*}

\appendix

\section{Hanson-Wright type inequality for isotropic log-concave random vector}\label{Sec. Hanson-Wright inequality}
In this section, we briefly reproduce Adamaczak's augument in \cite{Ada} to prove (\ref{121701}), under a different tail probability of the Lipschitz concentration. In the sequel, we fix $k = 1$ and let $\mathbf{y} := \sqrt{N}(x_{11},\cdots,x_{M1})^{\mathrm{T}}$. Thus, $\mathbf{y}$ is an isotropic log-concave random vector.

To prove (\ref{121701}), we  need the following technical lemmas. The proofs of these lemmas are essentially the same as those in \cite{Ada}, we modify them to meet the Lipschitz concentration in Corollary \ref{Lip concentration}.

\begin{lemma}\label{median to expectation} Assume that a random variable $Z$ satisfies 
	\begin{align}
		\mathbb{P} \left(| Z - \mathrm{Med}(Z)  | \ge t \right) \le 2\exp \Big( -\min \Big(\frac{t}{a}, \sqrt{ \frac{t}{b} }  \Big) \Big).
	\end{align}
	Here $\mathrm{Med}(Z)$ is the median of $Z$. Then for some absolute constant $C$ and all $t > 0$
	\begin{align}
		\mathbb{P} \left(| Z - \E (Z)  | \ge t \right) \le 2\exp \Big( -\min \frac{1}{C}\Big(\frac{t}{a}, \sqrt{ \frac{t}{b} }  \Big) \Big).
	\end{align}
\end{lemma}

\begin{lemma}\label{tech lemma}
	Let $Y$ and $Z$ be random variables and $a,b,t > 0$ be such that for all $s > 0$,
	\begin{align}\label{Y1}
		\mathbb{P}\left( |Y - \E (Y)| \ge s \right) \le 2 \exp \Big( \frac{-s}{a + \sqrt{bt}} \Big)
	\end{align}
	and 
	\begin{align}\label{Y2}
		\mathbb{P} (Y \neq Z) \le 2 \exp \Big( -\sqrt{ \frac{t}{b}} \Big).
	\end{align}
	Then
	\begin{align}
		\mathbb{P} \left( |Z - \mathrm{Med} (Z)| \ge t \right) \le 2 \exp \Big(-\frac{1}{C}\min \Big( \frac{t}{a}, \sqrt{ \frac{t}{b}} \Big)  \Big).
	\end{align}
\end{lemma}

 With the above lemmas, we proceed to the proof of (\ref{121701}). Letting $Z \equiv \varphi(\mathbf{y}) := \mathbf{y}^{\mathrm{T}}B\mathbf{y}$, the main part of the proof is to construct a suitable $Y$, such that (\ref{Y1}) and (\ref{Y2}) in Lemma \ref{tech lemma} hold. With the conclusion of Lemma \ref{tech lemma}, and further using Lemma \ref{median to expectation} to replace the median with mean, we finish the proof of (\ref{121701}).

Notice that if $Z$ is a Lipschitz function, we can set $Y = Z$. However, the quadratic function is not Lipschitz. The key idea in \cite{Ada} to construct $Y$ can be summarized as follows. First, observe that $\|\nabla \varphi(\mathbf{y})\|_2$ is a Lipschitz function. We can restrict $\mathbf{y}$ in a small convex set $\mathcal{B}$ where $\|\nabla \varphi(\mathbf{y})\|_2$ is bounded by an appropriately chosen constant, and then $\varphi(\mathbf{y})$ is Lipschitz on $\mathcal{B}$. Next, we can extend this restricted function to $\mathbb{R}^M$. This means that we can find another function $\tilde{\varphi}(\mathbf{y})$ which is Lipschitz on $\mathbb{R}^M$ and is equal to $\varphi(\mathbf{y})$ on $\mathcal{B}$. Then $Y$ can be set as $\tilde{\varphi}(\mathbf{y})$.

Without loss of generality, we prove (\ref{121701}) with real symmetric $B$. Notice that
\begin{align*}
	(\E \|\nabla\varphi(\mathbf{y})\|_2)^2 \le \E \|\nabla\varphi(\mathbf{y}) \|_2^2 = 4\|B \|_{HS}^2.
\end{align*}
This implies that $\E \|\nabla\varphi(\mathbf{y})\| \le 2 \| B\|_{HS}$. For any $t > 0$, define the set $\mathcal{B}$ to be $$\mathcal{B} := \{ \mathbf{y} \in \mathbb{R}^M: \|\nabla \varphi(\mathbf{y}) \|_2 \le 2\| B\|_{HS} + \sqrt{t \| B\|}  \},$$
and  $\tilde{\varphi}(\mathbf{y})$ is defined as
\begin{align*}
	\tilde{\varphi}(\mathbf{y}) := \max_{\mathbf{x} \in \mathcal{B}} ( \langle \nabla \varphi(\mathbf{x}), \mathbf{y}-\mathbf{x} \rangle + \varphi(\mathbf{x}) ).
\end{align*}
One can easily verify that $\tilde{\varphi}(\mathbf{y}) \le {\varphi}(\mathbf{y})$ on $\mathbb{R}^M$ and $\tilde{\varphi}(\mathbf{y}) = {\varphi}(\mathbf{y})$ on $\mathcal{B}$. Therefore,
\begin{align*}
	\mathbb{P}(\tilde{\varphi}(\mathbf{y}) \neq {\varphi}(\mathbf{y})) \le \mathbb{P}(\mathbf{y} \notin \mathcal{B}) = & \mathbb{P} (\|\nabla \varphi(\mathbf{y}) \|_2 > 2\| B\|_{HS} + \sqrt{t \| B\|} ) \\
	\le & \mathbb{P} (| \|\nabla \varphi(\mathbf{y}) \|_2  - \E\|\nabla \varphi(\mathbf{y}) \|_2 |  >  \sqrt{t \| B\|} ) \le \exp \left(- \frac{\psi_{\rho}}{C}  \sqrt{\frac{ t}{ \|B \|}}\right),
\end{align*}
where in the last inequality we used  Corollary \ref{Lip concentration} with the fact that $\|\nabla\varphi(\mathbf{y})\|$ is $2\|B\|$-Lipschitz. This established (\ref{Y2}) in Lemma \ref{tech lemma}. For (\ref{Y1}), one can verify that the Lipschitz constant of $\tilde{\varphi}(\mathbf{y})$ is $2\| B\|_{HS} + \sqrt{t\| B\|}$. Then by Corollary \ref{Lip concentration}, we get for any $s > 0$,
\begin{align*}
	\mathbb{P} \left( |\tilde{\varphi}(\mathbf{y}) - \E \tilde{\varphi}(\mathbf{y}) | \ge s \right) \le 2\exp \left( -\frac{\psi_\rho}{C} \cdot \frac{s}{\big(2\| B\|_{HS} + \sqrt{t\| B\|}\big)} \right),
\end{align*}
which established (\ref{Y1}).

\section{Proofs  of technical estimates in Section \ref{Sec. comparison}}

\subsection{Proofs of (\ref{GQBound})-(\ref{G2ijBound})}\label{sec Proofs of GQBound}
\begin{proof}[Proof of (\ref{GQBound})]
Recall that $z = E + \mathrm{i}\eta_0$ with $\eta_0 = N^{-2/3-\epsilon}$. By Cauchy's integral formula, we have
	\begin{align*}
		\langle x_1, (\mathcal{G}^{(1)}_1(z))^2x_1 \rangle =  \frac{1}{2\pi \mathrm{i}}\oint_{\gamma_{\eta_0}} \frac{\langle x_1, \mathcal{G}^{(1)}_1(a)x_1 \rangle  }{(a -z)^2}\mathrm{d}a,
	\end{align*}
	where $\gamma_{\eta_0}$ is chosen as the circle centered at $z$ with radius $\eta_0/2$. Therefore, we have by triangular inequality,
\begin{align*}
		\left|\langle x_1, (\mathcal{G}^{(1)}_1(z))^2x_1 \rangle\right| 
		\lesssim&  \left|\oint_{\gamma_{\eta_0}} \frac{ N^{-1}\Tr \mathcal{G}^{(1)}_1(a)  }{(a -z)^2}\mathrm{d}a \right| + \left|\oint_{\gamma_{\eta_0}} \frac{\langle x_1, \mathcal{G}^{(1)}_1(a)x_1 \rangle - N^{-1}\Tr \mathcal{G}_1^{(1)}(a)  }{(a -z)^2}\mathrm{d}a \right|
	\end{align*}
	For the first term,  we have
	\begin{align*}
		\left|\oint_{\gamma_{\eta_0}} \frac{ N^{-1}\Tr \mathcal{G}^{(1)}_1(a)  }{(a -z)^2}\mathrm{d}a \right| \lesssim \left|\frac{\Tr (\mathcal{G}^{(1)}_1(z))^2}{N}\right| \lesssim \frac{1}{N}\sum_{i=1}^N \frac{1}{|\lambda_i^{(1)} - z|^2} = \frac{\Im \Tr \mathcal{G}_2^{(1)}(z)}{N\eta_0} \prec \frac{1}{\sqrt{\eta_0}} \leq  N^{1/3+\epsilon}.
	\end{align*} 
	For the second term, using large deviation lemma (cf. Lemma \ref{Large deviation lemma}), we have
	\begin{align*}
		 \left|\oint_{\gamma_{\eta_0}} \frac{\langle x_1, \mathcal{G}^{(1)}_1(a)x_1 \rangle - N^{-1}\Tr \mathcal{G}^{(1)}_1(a)  }{(a -z)^2}\mathrm{d}a \right| \prec &  \left(\frac{\Im \Tr \mathcal{G}^{(1)}_1(z)}{N^2\eta_0^3}\right)^{1/2} \lesssim  \left( \frac{\Im \Tr \mathcal{G}_2^{(1)}(z)}{N^2\eta_0^3} \right)^{1/2} \prec N^{1/3+\epsilon}.
	\end{align*}
	Then (\ref{GQBound}) follows by combining the above estimates .
\end{proof}

\begin{proof}[Proof of (\ref{GijBound})]
	In the following proof, $z = E + \mathrm{i}\eta_0$ with $\eta_0 = N^{-2/3-\epsilon}$ and $|E - \lambda_+| \le N^{-2/3+\epsilon}$. For notational simplicity, We show the bound for $[\mathcal{G}_1(z)]_{11}$ while the bound for $[\mathcal{G}_1^{(1)}(z)]_{ii}$ can be proved similarly. By the resolvent identity, we have
\begin{align*}
	&[\mathcal{G}_1(z)]_{11} = \frac{1}{-z - z\langle r_1, \mathcal{G}_2^{[1]}(z)r_1\rangle} =:\frac{1}{-z - zm_2^{[1]}(z)-z(\mathsf{I}_1 + \mathsf{I}_2 + \mathsf{I}_3)},
\end{align*}
where
\begin{align*}
	&\mathsf{I}_1 := \Big(\sum_{i}x_{1i}^2 - 1\Big)m_2^{[1]}(z),\; \mathsf{I}_2 := \sum_{i}x_{1i}^2\left([\mathcal{G}_2^{[1]}(z)]_{ii} - m_2^{[1]}(z) \right),\; \mathsf{I}_3 := \sum_{i,j,i \neq j}x_{1i}x_{1j}[\mathcal{G}_2^{[1]}(z)]_{ij}. 
\end{align*}
Note that Theorem \ref{Strong MP law} can be still applied to $\mathcal{G}_2^{[1]}(z)$ since each column of $X^{[1]}$ is still an unconditional isotropic  log-concave random vector with dimensional $M-1$ (cf. Proposition \ref{iso prop}). Therefore, 
\begin{align*}
	\Big|[\mathcal{G}_2^{[1]}(z)]_{ij} - m_2^{[1]}(z)\delta_{ij} \Big|\prec  \sqrt{\frac{\Im m_2^{[1]}(z)}{N\eta} } + \frac{1}{N\eta}.
\end{align*}
Here $ m_2^{[1]}(z)$ is the Stieljes transform of Marcenko-Pastur law with parameter $\tilde{y} = (M-1)/N$.

We first consider the estimate of $\mathsf{I}_3$. Since $x_{1i}$'s are unconditionally distributed, we have
\begin{align*}
	\sum_{i,j,i \neq j}x_{1i}x_{1j}[\mathcal{G}_2^{[1]}(z)]_{ij} \overset{d}{=} \sum_{i,j,i \neq j}\delta_i \delta_j x_{1i}x_{1j}[\mathcal{G}_2^{[1]}(z)]_{ij},
\end{align*}
where $\delta_i$'s are i.i.d Rademacher random variables. Therefore, by large deviation of the quadratic form of $\delta_i$'s, we have
\begin{align*}
	\Big| \sum_{i,j,i \neq j}\delta_i \delta_j x_{1i}x_{1j}[\mathcal{G}_2^{[1]}(z)]_{ij} \Big| \prec \Big( \sum_{i,j,i\neq j}x_{1i}^2x_{1j}^2|[\mathcal{G}_2^{[1]}(z)]_{ij}|^2 \Big)^{1/2} \prec N^{-1/3+\epsilon}.
\end{align*}
Hence, we have 
$
	|\mathsf{I}_3| \prec N^{-1/3+\epsilon}.
$
For $\mathsf{I}_2$, we have
\begin{align*}
	|\mathsf{I}_2 |\le \sum_{i}x_{1i}^2|[\mathcal{G}_2^{[1]}(z)]_{ii} - m_2^{[1]}(z)| \prec N^{-1/3+\epsilon} \sum_{i} x_{1i}^2\prec N^{-1/3+\epsilon}.
\end{align*}
Lastly, for $\mathsf{I}_1$, using the fact that $|m_2^{[1]}(z)| \sim 1$, we can trivially bound it as  $|\mathsf{I}_1| \prec N^{-1/2}$ since $x_{1i}$'s are i.i.d.. Then using (\ref{identitym1m2}), the fact $|m_1(z)-m_1^{[1]}(z)|\lesssim \frac{1}{N\eta_0}$,  together with the above estimates, we have
\begin{align*}
	&[\mathcal{G}_1(z)]_{11}  = m_1(z)  + O_{\prec}(N^{-1/3+\epsilon}).
\end{align*}
For $i \neq j$, by the resolvent identity, we have,
\begin{align*}
	|[\mathcal{G}_1(z)]_{ij}| = |z[\mathcal{G}_1(z)]_{ii}[\mathcal{G}_1^{[i]}(z)]_{jj}\langle r_i , \mathcal{G}_2^{[ij]}(z) r_j\rangle| \prec | \langle r_i ,\mathcal{G}_2^{[ij]}(z) r_j\rangle| =  \Big| \sum_{k,l}x_{ik}x_{jl}[\mathcal{G}_2^{[ij]}(z)]_{kl} \Big|
\end{align*}
Again, by the unconditional assumption, we have
\begin{align*}
	\sum_{k,l}x_{ik}x_{jl}[\mathcal{G}_2^{[ij]}(z)]_{kl} \overset{d}{=} \sum_{k,l}\delta_{ik} \delta_{jl} x_{ik}x_{jl}[\mathcal{G}_2^{[ij]}(z)]_{kl}, 
\end{align*}
where $\delta_{ik}$'s and $\delta_{jl}$'s are i.i.d Rademacher random variables. Applying the large deviation inequality for the quadratic forms of $\delta$-variables, we have 
\begin{align*}
	\Big|\sum_{k,l}\delta_{ik} \delta_{jl} x_{ik}x_{jl}[\mathcal{G}_2^{[ij]}(z)]_{kl}\Big| \prec\Big(\sum_{k,l}x_{ik}^2x_{jl}^2 |[\mathcal{G}_2^{[ij]}(z)]_{kl}|^2\Big)^{1/2}  \prec N^{-1/3+\epsilon}. 
\end{align*}
Here we used the bounds from Theorem \ref{Strong MP law} for $\mathcal{G}_2^{[ij]}$. 
Therefore, we have the bound
$
	|[\mathcal{G}_1(z)]_{ij}| \prec  N^{-1/3+\epsilon}.
$ This concludes the proof of (\ref{GijBound}). 
\end{proof}

\begin{proof}[Proof of (\ref{G2ijBound})]
	By the Cauchy's integral formula, we have
	\begin{align*}
		|[(\mathcal{G}_1(z))^2]_{ij}| \lesssim \left|  \oint_{\gamma_{\eta_0}} \frac{[\mathcal{G}_1(a)]_{ij}}{(a - z)^2}\mathrm{d}a\right|.
	\end{align*}
	For the case of $i \neq j$, we can directly bound the RHS of the above inequality using (\ref{GijBound}),
	\begin{align*}
		\left|  \oint_{\gamma_{\eta_0}} \frac{[\mathcal{G}_1(a)]_{ij}}{(a - z)^2}\mathrm{d}a\right| \prec  N^{1/3+\epsilon}.
	\end{align*}
	If $i = j$, we have
	\begin{align*}
		\left|  \oint_{\gamma_{\eta_0}} \frac{[\mathcal{G}_1(a)]_{ii}}{(a - z)^2}\mathrm{d}a\right|  \lesssim &	\left|  \oint_{\gamma_{\eta_0}} \frac{[\mathcal{G}_1(a)]_{ii} -m_1(a)}{(a - z)^2}\mathrm{d}a\right|  + 	\left|  \oint_{\gamma_{\eta_0}} \frac{m_1(a)}{(a - z)^2}\mathrm{d}a\right| 
		\prec  \frac{N^{-1/3 + \epsilon}}{\eta_0} + |m_1'(z)| \prec  N^{1/3+\epsilon}.
	\end{align*}
	Hence, (\ref{G2ijBound}) follows from combining the above estimates.
\end{proof}

%

\end{document}